\documentclass[reqno,12pt]{amsart}

\usepackage{enumerate}
\usepackage[margin=1in]{geometry}
\usepackage{ifpdf}
\usepackage{amsmath}
\usepackage{amsfonts}
\usepackage{amssymb}
\usepackage{amsthm}
\usepackage[hyperfootnotes=false]{hyperref}
\usepackage{setspace}
\usepackage{amsrefs}
\usepackage[usenames]{xcolor}

\newcommand{\ignore}[1]{}

\renewcommand{\Re}{\operatorname{Re}}
\renewcommand{\Im}{\operatorname{Im}}

\newcommand{\abs}[1]{\left\lvert {#1} \right\rvert}

\newcommand{\C}{{\mathbb{C}}}
\newcommand{\R}{{\mathbb{R}}}

\newcommand{\sL}{{\mathcal{L}}}

\newcommand{\sV}{{\mathcal{V}}}
\newcommand{\sW}{{\mathcal{W}}}

\newtheorem{thm}{Theorem}[section]

\newtheorem{prop}[thm]{Proposition}

\newtheorem{cor}[thm]{Corollary}

\newtheorem{lemma}[thm]{Lemma}

\theoremstyle{definition}

\newtheorem{example}[thm]{Example}

\theoremstyle{remark}

\author{Xianghong Gong}
\address{Department of Mathematics, University of Wisconsin,
Madison, WI 53706, USA}
\email{gong@math.wisc.edu}
\author{Ji\v{r}\'i Lebl}
\thanks{The second author was in part supported by NSF grant DMS 0900885 and
DMS 1362337.}
\address{Department of Mathematics, Oklahoma State University,
Stillwater, OK 74078, USA}
\email{lebl@math.okstate.edu}
\date{October 21, 2014}

\ifpdf
\hypersetup{
  pdftitle={Normal forms for CR singular codimension two Levi-flat submanifolds},
  pdfauthor={Xianghong Gong and Jiri Lebl}
}
\fi

\title{Normal forms for CR singular codimension two Levi-flat submanifolds}

\begin{document}


\begin{abstract}
Real-analytic Levi-flat codimension two CR singular submanifolds
are a natural generalization to $\C^m$, $m > 2$, of Bishop surfaces in $\C^2$.
Such submanifolds for example arise as zero sets of mixed-holomorphic
equations with one variable antiholomorphic.  We classify the codimension
two Levi-flat CR singular quadrics, and we notice that new
types of submanifolds arise in dimension 3 or greater.  In fact, the
nondegenerate submanifolds, i.e.\ higher order purturbations of
$z_m=\bar{z}_1z_2+\bar{z}_1^2$, have no analogue in dimension 2.
We prove that
the Levi-foliation extends through the singularity in the real-analytic
nondegenerate case.
Furthermore, we prove that the quadric is a
(convergent) normal form
for a natural large class of such submanifolds, and we compute its
automorphism group.  In general, we find a formal normal form in $\C^3$
in the nondegenerate case that shows  infinitely many formal
invariants.
\end{abstract}

\maketitle



\section{Introduction} \label{section:intro}

Let $M \subset \C^{n+1}$ be a real submanifold.
A fundamental question in CR geometry is to classify $M$ at a point
up to local biholomorphic transformations.  One approach is
to find a normal form for $M$.

A real-analytic hypersurface $M \subset \C^{n+1}$ is Levi-flat if
the Levi-form vanishes identically.
Roughly speaking, a Levi-flat submanifold is a family of complex submanifolds.
Intuitively, a Levi-flat submanifold is as close
to a complex submanifold as possible.
In the real-analytic smooth
hypersurface case, it is well-known that $M$ can locally be transformed into
the real hyperplane given by
\begin{equation}
\Im z_1 = 0 .
\end{equation}
We therefore focus on higher codimension case, in particular on codimension
2.  A codimension 2 submanifold is again given by a single equation, but in
this case a complex valued equation.  A new phenomenon that appears in
codimension 2 is that $M$ may no longer be a CR submanifold.
Let $T_p^cM \subset T_pM$ be the largest subspace with
$J T_p^c M = T_p^c M$,
where $J$ is the complex structure on $\C^{n+1}$.
A submanifold is CR if $\dim T_p^cM$ is constant.

Real submanifolds of dimension $n+1$ in $\C^{n+1}$ with a non-degenerate complex tangent point
has been studied extensively after the fundamental work of
E.~Bishop~\cite{Bishop65}.  In $\C^2$, Bishop studied the submanifolds
\begin{equation}
w = z\bar{z} + \gamma (z^2+\bar{z}^2) + O(3)
\end{equation}
where $\gamma \in [0,\infty]$ is called the Bishop invariant, with $\gamma
= \infty$ interpreted as $w = z^2+ \bar{z}^2 + O(3)$.
One of Bishop's motivations was to
study the hull of holomorphy of the real submanifolds by attaching analytic discs.
Bishop's work 
on the family of attached analytic discs has been refined by Kenig-Webster
\cites{KenigWebster:82, KenigWebster:84}, Huang-Krantz~\cite{HuangKrantz95}, and
Huang~\cite{Huang:jams}.   The normal form theory for real submanifolds for
Bishop surfaces or submanifolds was established by
Moser-Webster~\cite{MoserWebster83}; see
also Moser~\cite{Moser85}, Gong~\cites{Gong04, Gong94:duke, Gong94:helv}, 
Huang-Yin \cite{HuangYin09}, and Coffman\cite{Coffman:unfolding}.
We would like to mention that the Moser-Webster normal form does not deal with the case of vanishing Bishop invariant.

The formal normal form and its application to holomorphic classification for surfaces with vanishing
Bishop invariant was achieved by Huang-Yin~\cite{HuangYin09}  by a completely different method. 
Real submanifolds with complex tangents have been studied in other situations. 
See for example \cite{LMSSZ}, where CR singular submanifolds that are
images of CR manifolds were studied.
Normal forms for the quadratic part of
general codimension two CR singular submanifolds in $\C^3$ was completely solved by
Coffman~\cite{Coffman:fourfolds}.  Huang and Yin~\cite{HuangYin09:codim2}
studied the normal form for codimension two CR singular submanifolds
of the form $w=\abs{z}^2 + O(3)$.
Dolbeault-Tomassini-Zaitsev~\cites{DTZ, DTZ2}
and Huang-Yin~\cite{HuangYin:flattening}
studied CR singular submanifolds of codimension two 
that are boundaries of Levi-flat hypersurfaces.
Burcea~\cite{Burcea} constructed the formal normal form for codimension 2 CR singular
submanifolds approximating a sphere.
Coffman~\cite{Coffman:crosscap} found an algebraic normal form for
nondegenerate CR singular manifolds in high codimension and one dimensional
complex tangent.

To motivate our work,  we observe 
 that in Bishop's work, the real submanifolds are Levi-flat away from their CR singular sets.
 Our purpose is to understand such submanifolds in higher dimensional
case with codimension being exactly two. Notice that the latter
is the smallest codimension for CR singularity to be present in (smooth) submanifolds. 
Regarding  CR singular
Levi-flat real codimension 2 submanifolds
on $\C^{n+1}$ as a natural generalization of Bishop surfaces to $\C^{n+1}$, 
we wish  to find their normal forms.
For singular Levi-flat hypersurfaces and related work on foliations with
singularity, see
~\cites{Bedford:flat,BG:lf,CerveauLinsNeto,Lebl:lfsing,FernandezPerez:gensing,Brunella:lf}. 

Our techniques revolve around the study of the Levi-map
(the generalization of the
Levi-form to higher codimension submanifolds)
of codimension 2 submanifolds.
Extending the CR structure through the singular point via Nash blowup
and then extending the Levi-map to this blowup has been studied previously by
Garrity~\cite{Garrity:BU}.

A CR
submanifold is \emph{Levi-flat} if the Levi-map
vanishes identically.  Locally, all
CR real-analytic Levi-flat submanifolds of real codimension 2 can be, after
holomorphic change of coordinates, written as
\begin{equation}
\Im z_1 = 0, \qquad \Im z_2 = 0 .
\end{equation}
If a submanifold $M$
is CR singular, denote by $M_{CR}$ the set of points where $M$ is CR.
We say $M$ is Levi-flat if $M_{CR}$ is Levi-flat in the usual sense.
A Levi-flat CR singular submanifold 
has no local biholomorphic invariants at the CR points, just as in
the case of Bishop surfaces.

A real, real-analytic
codimension 2 submanifold
that is CR singular at the origin can be written in coordinates
$(z,w) \in \C^{n} \times \C = \C^{n+1}$ as
\begin{equation}
w = \rho(z,\bar{z})
\end{equation}
for $\rho$ that is $O(2)$.
We will be concerned with submanifolds where the quadratic
part in $\rho$ is nonzero in any holomorphic coordinates.  We say that such
submanifolds have
a \emph{nondegenerate complex tangent}.  For example, the Bishop surfaces in
$\C^2$ are precisely the CR singular submanifolds with nondegenerate complex
tangent.

First, let us classify the quadratic parts of CR singular Levi-flats,
and in the process completely classify the CR singular Levi-flat quadrics,
that is those where $\rho$ is a quadratic.

\begin{thm} \label{thm:quadratic}
Suppose that $M \subset \C^{n+1}$, $n \geq 2$, is a germ of a
real-analytic
real codimension 2
submanifold, CR singular at the origin, written
in coordinates $(z,w) \in \C^{n} \times \C$ as
\begin{equation} \label{eq:generalform}
w = A(z,\bar{z}) + B(\bar{z},\bar{z}) + O(3),
\end{equation}
for quadratic $A$ and $B$, where $A+B \not\equiv 0$ (nondegenerate complex
tangent).  Suppose that $M$ is Levi-flat (that is $M_{CR}$
is Levi-flat).
\begin{enumerate}[(i)]
\item
If $M$ is a quadric, then $M$
is locally biholomorphically equivalent to one and exactly one of the
following:
\begin{equation} \label{eq:quadnormalforms}
\begin{aligned}
\text{(A.1)} \quad & w = \bar{z}_1^2 , \\
\text{(A.2)} \quad & w = \bar{z}_1^2 + \bar{z}_2^2, \\
& \vdots \\
\text{(A.$n$)} \quad & w = \bar{z}_1^2 + \bar{z}_2^2 + \dots + \bar{z}_{n}^2  , \\[10pt]
\text{(B.$\gamma$)} \quad & w = \abs{z_1}^2 +  \gamma\bar{z}_1^2 , ~~ \gamma \geq 0, \\[10pt]
\text{(C.0)} \quad & w = \bar{z}_1z_2 , \\
\text{(C.1)} \quad & w = \bar{z}_1z_2 + \bar{z}_1^2 .
\end{aligned}
\end{equation}
\item \label{thmitem:quadpartlf}
If $M$ is real-analytic, then the quadric
$w = A(z,\bar{z}) + B(z,\bar{z})$
is Levi-flat, and can be put via a biholomorphic transformation into
exactly one of the forms \eqref{eq:quadnormalforms}.
\end{enumerate}
\end{thm}

By part \eqref{thmitem:quadpartlf}, the quadratic part in
\eqref{eq:generalform} is an invariant
of $M$ at a point.  We say the \emph{type} of $M$ at the origin
is A.x, B.$\gamma$, or C.x depending on the type of the quadratic form.
Following Bishop, we call types B.$\gamma$ and A.1 Bishop-like,
we could think of $\gamma=\infty$ as A.1.

By type being \emph{stable} we mean that the type does
not change at all complex tangents in a neighborhood of the origin under any
small (or higher order)  perturbations that stay within the class of
Levi-flat CR singular submanifolds.
As a consequence of the above theorem and
because rank is lower semicontinuous, we get that the only types that
are stable are A.$n$ and C.1, although A.$n$ are
degenerate because the form $A(z,\bar{z})$ is identically zero.
See also Proposition~\ref{prop:instability}.

The quadrics A.$k$ for $k \geq 2$ do not possess a
nonsingular foliation extending the Levi-foliation of $M_{CR}$
through the origin.  In fact, there is a singular
complex subvariety of dimension 1 through the origin contained in $M$.
See \S~\ref{sec:foliation}.

In the sequel, when we wish to refer to the quadric of certain
type we will use the notation $M_{C.1}$ to denote the quadric
of type C.1.

The quadratic form $A(z,\bar{z})$ carries the ``Levi-map'' of the
submanifold.
Type C.1 is the unique quadric that is stable and has non-zero $A$.
Having non-zero $A$ is also stable in a neighborhood of the origin under
any small (or higher order) perturbations.
Therefore, we say a type is \emph{non-degenerate} if it is C.1
and we focus mostly on such  submanifolds. 
First, we show that submanifolds of type C.x possess
a nonsingular real-analytic foliation that extends the Levi-foliation, due
to the form $A(z,\bar{z})$:

\begin{thm} \label{thm:folextendsCxtype}
Suppose that $M \subset \C^{n+1}$, $n \geq 2$, is a real-analytic
Levi-flat CR singular submanifold of type C.1 or C.0,
that is, $M$ is given by
\begin{equation}
w = \bar{z}_1z_2 + \bar{z}_1^2 + O(3)
\qquad \text{or}
\qquad
w = \bar{z}_1z_2 + O(3).
\end{equation}
Then there exists a nonsingular real-analytic foliation defined
on $M$ that extends
the Levi-foliation on $M_{CR}$, and consequently, there exists
a CR real-analytic mapping $F \colon U \subset \R^2 \times \C^{n-1} \to
\C^{n+1}$ such that $F$ is a diffeomorphism onto $F(U) = M \cap U'$,
for some neighbourhood $U'$ of 0.
\end{thm}

Here the CR structure on $\R^2\times\C^{n-1}$ is
induced from $\C^2\times\C^{n-1}$.
As a corollary of this theorem we obtain in \S~\ref{sec:crsingc1}
using the results of \cite{LMSSZ} that the
CR singular set of any type C.1 submanifold is
a Levi-flat submanifold of dimension $2n-2$ and CR dimension $n-2$.

The Levi-foliation on a type C.x
submanifold cannot extend to a whole neighbourhood of $M$ as a nonsingular
holomorphic foliation.  If it did, we could flatten the foliation and
$M$ would be a Cartesian product, in particular Bishop-like.
Thus, the study of normal form theory for the special case when the
foliation extends to a neighbourhood is reduced to
the case of Bishop surfaces, which have been studied extensively.
 
A codimension $2$ submanifold in $\mathbb{C}^m$ can arise from
\begin{equation} \label{eq:mixhol}
f(\bar{z}',z'')=0
\end{equation}
for a suitable holomorphic function $f$ in $m$ variables.
The zero set admits two holomorphic foliations. We are interested in
the case where one of foliations has leaves of maximum dimension $m-2$, while the other
has leaves of minimum dimension $0$. Therefore, we will assume that $z'=z_1$
and $z''=(z_2,\ldots, z_m)$.
Functions holomorphic in some variables and anti-holomorphic in other
variables, such as \eqref{eq:mixhol}, are often called \emph{mixed-holomorphic} or mixed-analytic,
and come up often in complex geometry, the simplest example being the
standard inner product.
An interesting feature of the mixed-holomorphic
setting is that the equation can be
complexified into $\C^m$, so the sets share some of the properties of
complex varieties.  However, they have a different automorphism group if we
wish to classify them under biholomorphic transformations.  Such
mixed-analytic sets are automatically real codimension 2, are
Levi-flat or complex, and may have CR singularities.  We study their
normal form in \S~\ref{sec:onebar}.  See also
Theorem~\ref{thm:themixedholform} below.

When a type C.1 CR singular submanifold has a defining equation
that does not depend on $\bar{z}_2, \ldots, \bar{z}_n$ we prove that
it is automatically Levi-flat, and it is equivalent to $M_{C.1}$.

\begin{thm} \label{thm:themixedholform}
Let $M \subset \C^{n+1}$, $n \geq 2$, be a real-analytic submanifold given by
\begin{equation} \label{eq:C1mixedhol}
w = \bar{z}_1 z_2 + \bar{z}_1^2 + r(z_1,\bar{z}_1,z_2,z_3,\ldots,z_n) ,
\end{equation}
where $r$ is $O(3)$.  Then $M$ is Levi-flat and
at the origin $M$ is locally biholomorphically equivalent to the quadric
$M_{C.1}$ submanifold
\begin{equation} \label{eq:C1mixedholnormform}
w = \bar{z}_1z_2 + \bar{z}_1^2 .
\end{equation}
\end{thm}

The theorem is also true formally; given a formal submanifold
of the form \eqref{eq:C1mixedhol}, it is formally equivalent to $M_{C.1}$.

A key idea in the proof of the convergence of the normalizing transformation
is that the form $B(\bar{z},\bar{z}) = \bar{z}_1^2$ induces
a natural mixed-holomorphic involution on quadric $M_{C.1}$.
This involution also plays a key role in computing the automorphism group
of the quadric in Theorem~\ref{thm:autogroup}.

%

Finally, we also compute the automorphism group for the quadric $M_{C.1}$, see
Theorem~\ref{thm:autogroup}.
In particular we show that
the automorphism
group is infinite dimensional.


Not every type C.1 Levi-flat submanifold is
biholomorphically equivalent to the C.1 quadric.
We will find a formal normal form for type C.1 Levi-flat submanifolds in
$\C^3$ that shows infinitely many formal invariants.  Let us give a
simplified statement.  For details see
Theorem~\ref{thm:formalnormformC3}.

\begin{thm}
Let $M$ be a
real-analytic Levi-flat type C.1 submanifold in $\C^{3}$. There exists a formal biholomorphic map transforming $M$ into 
 the image of
\begin{equation}
\hat\varphi(z,\bar{z},\xi)=\bigl(z+A(z,\xi, w)w\eta, \xi,w\bigr)
\end{equation}
with  $\eta=\bar z+\frac{1}{2}{\xi}$ and $w=\bar z\xi+\bar z^2$. Here $A=0$,
or $A$ satisfies certain normalizing conditions.

When $A \not= 0$ the formal automorphism group preserving the normal form is
finite or $1$ dimensional. 
\end{thm}

We do not know if the formal normal form above can be achieved by convergent
transformations, even if $A=0$.


\section{Invariants of codimension 2 CR singular submanifolds}

Before we impose the Levi-flat condition, let us find some invariants
of codimension two CR singular submanifolds in $\C^{n+1}$ with
CR singularity at 0.
Such a
submanifold can locally near the origin be put into the form
\begin{equation} \label{theeq}
w = A(z,\bar{z}) + B(\bar{z},\bar{z}) + O(3),
\end{equation}
where $(z,w) \in \C^{n} \times \C$ and $A$ and $B$ are quadratic forms.  
We think of $A$ and $B$ as matrices and $z$ as a column vector and
write the forms as
$z^*Az$ and $z^* B\bar{z}$ respectively.
The matrix $B$ is not unique.
Hence we make $B$ symmetric to make
the choice of the matrix $B$ canonical.
The following proposition is not difficult and
well-known.  Since the details are important and will be used later,
let us prove this fact.

\begin{prop} \label{prop:astensors}
A biholomorphic transformation of \eqref{theeq}
taking the origin to itself and
preserving the form of \eqref{theeq} takes the matrices
$(A,B)$ to
\begin{equation}
(\lambda T^* A T, \lambda T^* B \overline{T} ) ,
\end{equation}
for $T \in GL_n(\C)$ and $\lambda \in \C^*$.
If $(F_1,\ldots,F_n,G) = (F,G)$ is the transformation
then the linear part of $G$ is $\lambda^{-1} w$ and the linear part
of $F$ restricted to $z$ is $Tz$.
\end{prop}

Let us emphasize that $A$ is an arbitrary complex matrix and
$B$ is a symmetric, but not necessarily Hermitian, matrix.

\begin{proof}
Let $(F_1,\ldots,F_n,G) = (F,G)$ be a change of coordinates taking
\begin{equation}
w = \widetilde{A}(z,\bar{z}) + \widetilde{B}(\bar{z},\bar{z}) + O(3) =
\rho(z,\bar{z})
\end{equation}
to
\begin{equation}
w = A(z,\bar{z}) + B(\bar{z},\bar{z}) + O(3) .
\end{equation}
Then
\begin{equation} \label{eq:quadtermsplugged}
G\bigl(z,\rho(z,\bar{z})\bigr) =
A\Bigl(F\bigl(z,\rho(z,\bar{z})\bigr),
\bar{F}\bigl(\bar{z},\bar{\rho}(\bar{z},z)\bigr)\Bigr)
\\
+
B\Bigl(\bar{F}\bigl(\bar{z},\bar{\rho}(\bar{z},z)\bigr),
\bar{F}\bigl(\bar{z},\bar{\rho}(\bar{z},z)\bigr)\Bigr) + O(3)
\end{equation}
is true for all $z$.  The right hand side has no linear terms,
so the linear terms in $G$ do not depend on $z$.
That is, $G = \lambda^{-1} w + O(2)$, where $\lambda$ is a nonzero scalar
and the negative power is for convenience.

Let
$T = [ T_1, T_2 ]$ denote the matrix representing the linear terms
of $F$.  Here $T_{1}$ is an $n\times n$ matrix and $T_{2}$ is $n \times 1$.
Since the linear terms in $G$ do not depend on any $z_j$, 
$T_1$ is nonsingular.
Then the quadratic terms in \eqref{eq:quadtermsplugged} are
\begin{equation}
\lambda^{-1} \bigl(
\widetilde{A}(z,\bar{z}) + \widetilde{B}(\bar{z},\bar{z}) \bigr)
=
z^* T_{1}^* A T_{1} z +
z^* T_1^* B \overline{T}_{1} \bar{z} .
\end{equation}
In other words as matrices,
\begin{equation}
\widetilde{A} = \lambda T_{1}^* A T_{1} \qquad \text{and} \qquad
\widetilde{B} = \lambda T_1^* B \overline{T}_1 . \text{\qedhere}
\end{equation}
\end{proof}


We will need to at times reduce to the 3-dimensional case, and so we need
the following lemma.

\begin{lemma} \label{lemma:restriction}
Let $M \subset \C^{n+1}$, $n \geq 3$, be a real-analytic Levi-flat CR singular submanifold
of the form
\begin{equation}
w = A(z,\bar{z}) + B(\bar{z},\bar{z}) + O(3) ,
\end{equation}
where $A$ and $B$ are quadratic.
Let $L$ be a nonsingular $(n-2) \times n$ matrix $L$.  If
$A+B$ is not zero on the set $\{ L z = 0 \}$, then the submanifold
\begin{equation}
M_L = M \cap \{ L z = 0 \}
\end{equation}
is a Levi-flat CR singular submanifold.
\end{lemma}

\begin{proof}
%
%
Clearly if
$M_L$ is not contained in the CR singularity of $M$, then $M_L$ is a
Levi-flat CR singular submanifold.
$M_{L'}$ is not contained in the CR singularity of $M$ for a dense open
subset of $(n-2) \times n$ matrices $L'$.  If $M_L$ is a subset of the CR
singularity of $M$, pick a CR point $p$ of $M_L$ then pick a sequence $L_n$
approaching $L$ such that $M_{L_n}$ are not contained in the CR singularity
of $M$.  As $A+B$ is not zero on the set $\{ L z = 0 \}$, then $M_L$ is
not a complex submanifold, and therefore a CR singular submanifold.
Then as the Levi-form of $M_{L_n}$ vanishes at all CR points of
$M_{L_n}$, the Levi-form of $M_L$ vanishes at $p$, so $M_L$ is Levi-flat.
\end{proof}


\section{Levi-flat quadrics}

Let us first focus on Levi-flat quadrics.  We will prove later that
the quadratic part of a Levi-flat submanifold is Levi-flat.
Let $M$ be defined in
$(z,w) \in \C^{n} \times \C$
by
\begin{equation}
w = A(z,\bar{z}) + B(\bar{z},\bar{z}) .
\end{equation}
Being Levi-flat has
several equivalent formulations.  The main idea is that the $T^{(1,0)} M \times
T^{(0,1)} M$ vector fields are completely integrable at CR points and we obtain a foliation
of $M$ at CR points by complex submanifolds of complex dimension
$n-1$.  An equivalent notion is that the Levi-map  is identically zero, see
\cite{BER:book}.  The Levi-map for a CR submanifold defined by two real equations
$\rho_1 = \rho_2 = 0$ (for $\rho_1$ and $\rho_2$ with linearly independent differentials)
is the pair of Hermitian forms
\begin{equation}
i \partial \bar{\partial} \rho_1
\quad \text{and} \quad
i \partial \bar{\partial} \rho_2 ,
\end{equation}
applied to $T^{(1,0)} M$ vectors.
The full quadratic forms
$i \partial \bar{\partial} \rho_1$ and $i \partial \bar{\partial} \rho_2$
of course depend on the defining equations themselves and are
therefore extrinsic information.  It is important to note
that for the Levi-map we restrict it to 
$T^{(1,0)} M$ vectors.
We can define these two forms
$i \partial \bar{\partial} \rho_1$ and $i \partial \bar{\partial} \rho_2$
even at a CR singular point $p \in M$.

These forms are the
complex Hessian matrices of the defining equations.  For our quadric
$M$ they are the real and imaginary parts of the $(n+1) \times (n+1)$ complex
matrix
\begin{equation}
\widetilde{A} =
\begin{bmatrix}
A & 0 \\
0 & 0
\end{bmatrix} ,
\end{equation}
where the variables are ordered as $(z_1,\ldots,z_n,w)$.

For $M$ to be Levi-flat,
the quadratic form defined by $\widetilde{A}$
has to be zero when restricted to the $n-1$
dimensional
space spanned by $T^{(1,0)}_p M$ for every $p \in M_{CR}$.
In other words for every $p \in M_{CR}$
\begin{equation}
v^* \widetilde{A} v = 0, \qquad \text{for all $v \in T^{(1,0)}_pM$}.
\end{equation}
The space $T^{(1,0)}_pM$ is of dimension $n-1$, and furthermore,
the vector $\frac{\partial}{\partial w}$ is not in $T^{(1,0)}_pM$.
Therefore, $z^* A z = 0$ for $z \in \C^n$ in a subspace of
dimension $n-1$.

Before we proceed let us note the following general fact about CR singular
Levi-flat submanifolds.

\begin{lemma} \label{lemma:dflimit}
Suppose that $M \subset \C^{n+1}$, $n \geq 2$,
is a Levi-flat connected real-analytic real codimension
2 submanifold, CR singular at the origin.  Then there exists
a germ of a complex analytic variety of complex dimension $n-1$
through the origin, contained in $M$.
\end{lemma}

\begin{proof}
Through each point of $M_{CR}$ there exists a germ of a complex variety
of complex dimension $n-1$ contained in $M$.  The set
of CR points is dense in $M$.  Take a sequence $p_k$ of CR points
converging to the origin and take complex varieties of dimension $n-1$,
$W_k \subset M$ with $p_k \in W_k$.  A theorem of Forn{\ae}ss
(see Theorem 6.23 in \cite{kohn:subell} for a proof using
the methods of Diederich and Forn{\ae}ss \cite{DF}) implies that there exists
a variety through $W \subset M$ with $0 \in W$ and of complex dimension at least
$n-1$.
\end{proof}

Let us first concentrate on $n=2$.
When $n=2$, $T^{(1,0)} M$ is one dimensional
at CR points.
Write
\begin{equation}
A =
\begin{bmatrix}
a_{11} & a_{12} \\
a_{21} & a_{22}
\end{bmatrix}
,
\qquad
B =
\begin{bmatrix}
b_{11} & b_{12} \\
b_{12} & b_{22}
\end{bmatrix}
.
\end{equation}
Note that $B$ is symmetric.
A short computation
shows that the vector field can be
written as
\begin{equation}
\alpha \frac{\partial}{\partial w} +
\beta_1 \frac{\partial}{\partial z_1} +
\beta_2 \frac{\partial}{\partial z_2}
=
\alpha \frac{\partial}{\partial w} +
\beta \cdot \frac{\partial}{\partial z} ,
\end{equation}
where
\begin{equation}
\begin{aligned}
& \beta_1 = \bar{a}_{21}\bar{z}_1 + \bar{a}_{22}\bar{z}_2
+ 2 \bar{b}_{12}z_1 + 2 \bar{b}_{22} z_2 , \\
& \beta_2 = - \bar{a}_{11}\bar{z}_1 - \bar{a}_{12}\bar{z}_2
- 2 \bar{b}_{11}z_1 - 2 \bar{b}_{12} z_2 , \\
& \alpha = a_{11} \bar{z}_1 \beta_1 + a_{21} \bar{z}_2 \beta_1
+ a_{12} \bar{z}_1 \beta_2 + a_{22} \bar{z}_2 \beta_2  .
\end{aligned}
\end{equation}
Note that since the CR singular set is defined by $\beta_1=\beta_2=0$,  then $M_{CR}$ is dense in $M$.
Thus we need to check that
\begin{equation}
\begin{bmatrix}
\beta^* & \bar{\alpha}
\end{bmatrix}
\begin{bmatrix}
A & 0 \\
0 & 0
\end{bmatrix}
\begin{bmatrix}
\beta \\
\alpha
\end{bmatrix}
=
\beta^* A \beta
\end{equation}
is identically zero for $M$ to be Levi-flat.

If $A$ is the zero matrix, then $M$ is automatically Levi-flat.  We 
diagonalize $B$ via $T$ into a diagonal matrix with ones and zeros on
the diagonal.
We obtain (recall $n=2$) the submanifolds:
\begin{equation}
\begin{aligned}
& w = \bar{z}_1^2 , \qquad \qquad \text{or}\\
& w = \bar{z}_1^2 + \bar{z}_2^2 .
\end{aligned}
\end{equation}
The first submanifold is of the form $M \times \C$
where $M \subset \C^2$ is a Bishop
surface.

Let us from now on suppose that $A\not=0$.

As $M$ is Levi-flat, then
through each CR point $p = (z_p,w_p) \in M_{CR}$ we have a
complex submanifold of dimension 1 in $M$.
It is well-known that this submanifold is contained in
the Segre variety (see also \S~\ref{sec:quadsegre})
\begin{equation}
w = A(z,\bar{z}_p) + B(\bar{z}_p,\bar{z}_p), \qquad
\bar{w}_p = \overline{A}(\bar{z}_p,z) + \overline{B}(z,z) .
\end{equation}
By Lemma~\ref{lemma:dflimit} 
we obtain a complex variety $V \subset M$ of dimension
one through the origin.
Suppose without loss of generality that $V$ is irreducible.
$V$ has to be contained in the Segre
variety at the origin, in particular $w=0$ on $V$.  Therefore, to
simplify notation,
let us consider $V$ to be subvariety of $\{ w = 0 \}$.
Denote by $\overline{V}$ the complex conjugate of $V$.  Then as $V$ is
irreducible, then $V \times \overline{V}$ is also irreducible (the smooth
part of $V$ is connected and so the smooth part of $V \times \overline{V}$
is connected, see
\cite{Whitney:book}).
Hence, by complexifying, we have
$A(z,\bar{\xi}) + B(\bar{\xi},\bar{\xi}) = 0$ for all $z \in V$ and $\xi \in V$.


If $B \not= 0$, then setting $z=0$, we have
$B(\bar{\xi},\bar{\xi}) = 0$ on $V$.  As $B$ is homogeneous and $V$
is irreducible, $V$ is a one dimensional complex line.
If $B=0$, then
$A(z,\bar{\zeta})=0$ for $z,\zeta\in V$ as mentioned above.
We consider two cases. Suppose first that every $\sum_{j=1}^2 a_{ij}\bar{\zeta}_j$ is identically zero
 for all $\zeta\in V$ and $i=1$ and $i=2$. Then $V$ is contained in some  complex line $\sum_{j=1}^2\bar{a}_{ij}\zeta_j=0$.
 Suppose now that $A(z,\bar{\zeta}_*)$ is not identically zero for some
$\zeta_* \in V$. Then $V$ is contained in the complex line
 $A(z,\bar{\zeta}_*)=0$.  This shows that $V$ is a complex line.

Thus as $A(z,\bar{z}) + B(\bar{z},\bar{z})$ is zero on a one dimensional
linear subspace, we make this subspace
$\{ z_1 = 0 \}$ and so each monomial
in $A(z,\bar{z}) + B(\bar{z},\bar{z})$
is divisible by either $z_1$ or $\bar{z}_1$.
Therefore, $A$ and $B$ are matrices of the form
\begin{equation}
\begin{bmatrix}
* & * \\
* & 0
\end{bmatrix} ,
\end{equation}
that is $a_{22} = 0$ and $b_{22} = 0$.

To normalize the pair $(A,B)$,
we apply arbitrary invertible transformations $(T,\lambda) \in GL_{n}(\C) \times
\C^*$
as
\begin{equation}
(A,B) \mapsto (\lambda T^* A T,\lambda T^* B\overline{T}) .
\end{equation}


Recall that we are assuming that $A\not=0$.
If $a_{21} = 0$ or $a_{12} = 0$, then $A$ is rank one and via
a transformation $T$ of the form
\begin{equation}\label{z1z1}
z_1'=z_1, \quad z_2'=z_2+cz_1
\qquad \text{or} \qquad
z_2'=z_1, \quad z_1'=z_2+cz_1
\end{equation}
and rescaling by nonzero $\lambda$, the matrix
$A$ can be put in the form
\begin{equation}
\begin{bmatrix}
0 & 1 \\
0 & 0
\end{bmatrix} ,
\qquad \text{or} \qquad
\begin{bmatrix}
1 & 0 \\
0 & 0
\end{bmatrix} .
\end{equation}
The transformation $T$ and $\lambda$ must also be applied to $B$
and this could possibly make $b_{22} \not= 0$.
However, we will show that we actually have $b_{22}=0$. Thus $B=0$ on $z_1=0$ still holds true.

Let us first focus on
\begin{equation}
A =
\begin{bmatrix}
1 & 0 \\
0 & 0
\end{bmatrix} .
\end{equation}
We apply the $T^{(1,0)}$ vector field we computed above.
Only $a_{11}$ is nonzero in $A$.
Therefore $\beta^* A \beta$, which
must be identically zero, is
\begin{multline}
0 = \beta^* A \beta =
\bar{\beta}_1 \beta_1 =
\overline{
(
2 \bar{b}_{12}z_1 + 2 \bar{b}_{22} z_2
)
}
(
2 \bar{b}_{12}z_1 + 2 \bar{b}_{22} z_2
)
\\
=
4
(
\abs{b_{12}}^2z_1\bar{z}_1 +
\abs{b_{22}}^2z_2\bar{z}_2 +
b_{12}\bar{b}_{22}\bar{z}_1z_2 +
\bar{b}_{12}b_{22}z_1\bar{z}_2
) .
\end{multline}
This polynomial must be identically zero and hence all coefficients must
be identically zero.  So $b_{12} = 0$ and $b_{22}=0$.
In other words, only $b_{11}$ in $B$ can
be nonzero, in which case we make it nonnegative via a diagonal $T$ to obtain
the quadric
\begin{equation}
w = \abs{z_1}^2 +  \gamma\bar{z}_1^2 , \quad \gamma \geq 0 .
\end{equation}

Next let us focus on
\begin{equation}
A =
\begin{bmatrix}
0 & 1 \\
0 & 0
\end{bmatrix} .
\end{equation}
As above, we compute
$\beta^* A \beta$:
\begin{multline}
0 = \beta^* A \beta =
\bar{\beta}_1 \beta_2 =
\overline{
(
2 \bar{b}_{12}z_1 + 2 \bar{b}_{22} z_2
)
}
(
- \bar{z}_2
- 2 \bar{b}_{11}z_1 - 2 \bar{b}_{12} z_2
)
\\
=
- 2b_{12} \bar{z}_1\bar{z}_2
- 2 b_{22} \bar{b}_{11} z_1\bar{z}_2
- 4 \bar{b}_{11}b_{12}z_1\bar{z}_1 - 4 b_{12}\bar{b}_{12} \bar{z}_1 z_2
- 2b_{22}\bar{z}_2^2 -4 b_{22}\bar{b}_{12}  z_2\bar{z}_2.
\end{multline}
Again, as this polynomial must be identically zero, all coefficients
must be zero.  Hence $b_{12} = 0$ and $b_{22} = 0$.
Again only $b_{11}$ is left possibly nonzero.

Suppose that $b_{11} \not= 0$.  Then let $s$ be such that
$b_{11} \bar{s}^2 = 1$, and let $\bar{t} = \frac{1}{\bar{s}}$.
The matrix $T = \left[ \begin{smallmatrix} s & 0 \\ 0 & t \end{smallmatrix}
\right]$ is such that $T^* A T = A$
and
 $T^* B \overline{T} = \left[ \begin{smallmatrix} 1 & 0 \\ 0 & 0
\end{smallmatrix} \right]$.
If $b_{11} = 0$, we have $B=0$.  Therefore we have obtained two
distinct possibilities for $B$, and thus
the two submanifolds
\begin{equation}
\begin{aligned}
& w = \bar{z}_1z_2 , \qquad \qquad \text{or} \\
& w = \bar{z}_1z_2 + \bar{z}_1^2 .
\end{aligned}
\end{equation}
We emphasize that after $A$ is normalized by a transformation of the form
\eqref{z1z1}, only one coordinate change is needed
to normalize $b_{11}$ and this coordinate change preserves $A$. Both are required in a reduction proof for higher dimensions.

We have handled the rank one case.
Next we focus on the rank two case, that is $a_{21} \not= 0$
and $a_{12} \not=0$ (recall $a_{22} = 0$).
%
We normalize (rescale) $A$ to have $a_{12} = 1$ and take
\begin{equation}
A=
\begin{bmatrix}
a_{11} & 1 \\
a_{21} & 0
\end{bmatrix} .
\end{equation}
Again, let us compute $\beta^* A \beta$.  In the computation for
the rank 2 case, 
recall that we have not
done any normalization other than rescaling, so we can safely still assume that $b_{22} = 0$:
\begin{multline}
0 = \beta^* A \beta =
a_{11}\bar{\beta}_1 \beta_1 +
\bar{\beta}_1 \beta_2 +
a_{21}
\beta_1 \bar{\beta}_2
\\
=
a_{11}\overline{
(\bar{a}_{21}\bar{z}_1
+ 2 \bar{b}_{12}z_1
)
}
(\bar{a}_{21}\bar{z}_1
+ 2 \bar{b}_{12}z_1
)
+
\overline{
(\bar{a}_{21}\bar{z}_1
+ 2 \bar{b}_{12}z_1
)
}
(- \bar{a}_{11}\bar{z}_1 - \bar{z}_2
- 2 \bar{b}_{11}z_1 - 2 \bar{b}_{12} z_2
)
\\
+
a_{21}
\overline{
(- \bar{a}_{11}\bar{z}_1 - \bar{z}_2
- 2 \bar{b}_{11}z_1 - 2 \bar{b}_{12} z_2
)
}
(\bar{a}_{21}\bar{z}_1
+ 2 \bar{b}_{12}z_1
)
\\
=
(-4 \abs{b_{12}}^2 - \abs{a_{21}}^2) \bar{z}_1 z_2
+ \text{(other terms)}.
\end{multline}
All coefficients must be zero. So
$a_{21}=0$, and $A$ would not be rank 2.

Let us now focus on $n > 2$.  First let us suppose that $A=0$.  Then
as before $M$ is automatically Levi-flat and
by diagonalizing $B$ we obtain the $n$ distinct submanifolds:
\begin{equation}
\begin{aligned}
w & = \bar{z}_1^2 , \\
w & = \bar{z}_1^2 + \bar{z}_2^2 , \\
& ~\vdots \\
w & = \bar{z}_1^2 + \bar{z}_2^2 + \dots + \bar{z}_{n}^2  .
\end{aligned}
\end{equation}

Thus suppose from now on that $A \not=0$.
As before we have an irreducible $n-1$ dimensional
variety $V \subset M$ through the origin, such that
$w = 0$ and
$A(z,\bar{z}) + B(\bar{z},\bar{z}) = 0$ on $V$.

We wish to show that
$A(z,\bar{z}) + B(\bar{z},\bar{z}) = 0$ on an $n-1$ dimensional linear subspace.
For any $\xi \in V$ we obtain
$A(z,\bar{\xi}) + B(\bar{\xi},\bar{\xi}) = 0$ for all $z \in V$.
If $V$ is contained
in the kernel of the matrix $A^*$, then we have that $V$ is
a linear subspace of dimension $n-1$.
So suppose that $\bar{\xi}$ is
not in the kernel of the matrix $A^t$.  Then for a fixed $\bar{\xi}$
we obtain a linear
equation $A(z,\bar{\xi}) + B(\bar{\xi},\bar{\xi}) = 0$ for $z \in V$.

Therefore, as
$A(z,\bar{z}) + B(\bar{z},\bar{z})$ needs to be zero on an $n-1$ dimensional
subspace we can just make this $\{ z_1 = 0 \}$ and so each monomial
is divisible by either $z_1$ or $\bar{z}_1$.
Therefore, $A$ and $B$ is of the form
\begin{equation} \label{eq:initialformofAB}
\begin{bmatrix}
* & * & \cdots & * \\
* & 0 & \cdots & 0 \\
\vdots & \vdots & \ddots & \vdots \\
* & 0 & \cdots & 0
\end{bmatrix} ,
\end{equation}
that is, only first column and first row are nonzero.
We normalize $A$ via
\begin{equation}
(A,B) \mapsto (\lambda T^* A T,\lambda T^* B\overline{T}) ,
\end{equation}
as before.
We use column operations
on all but the first column
to make all but the first two columns have nonzero elements.  Similarly
we can do row operations on all but the first two rows and to make all but
first three rows nonzero.  That is $A$ has the form
\begin{equation}
\begin{bmatrix}
* & * & 0 & \cdots & 0 \\
* & 0 & 0 & \cdots & 0 \\
* & 0 & 0 & \cdots & 0 \\
0 & 0 & 0 & \cdots & 0 \\
\vdots & \vdots & \vdots & \ddots & \vdots \\
0 & 0 & 0 & \cdots & 0
\end{bmatrix} .
\end{equation}

By Lemma~\ref{lemma:restriction}, setting $z_3 = \cdots = z_n = 0$ we obtain a Levi-flat submanifold where
the matrix corresponding to $A$ is the principal $2 \times 2$ submatrix of
$A$.  This submatrix cannot be of rank $2$ and hence either $a_{12} = 0$
or $a_{21} = 0$.
If $a_{21} = 0$ and $a_{12} \not= 0$, then setting $z_2 = z_3$, and $z_4 =
\cdots = z_n = 0$ we again must have a rank one matrix and therefore
$a_{31} = 0$.

Therefore, if $a_{12} \not= 0$ then all but $a_{11}$ and $a_{12}$ are zero.
If $a_{12} = 0$, then via a further linear map not involving $z_1$ we can
ensure that $a_{31} = 0$.  In particular, $A$ is of rank 1 and
can only be nonzero in the principal $2 \times 2$ submatrix.
At this point $B$ is still of the form
\eqref{eq:initialformofAB}.

Via a linear change of coordinates in the first two variables, the principal 
$2 \times 2$ submatrix of $A$ can be normalized into one of the 2 possible forms
\begin{equation}
\begin{bmatrix}
1 & 0 \\
0 & 0
\end{bmatrix} ,
\qquad \text{or} \qquad
\begin{bmatrix}
0 & 1 \\
0 & 0
\end{bmatrix} .
\end{equation}
Recall that $A=0$ was already handled.

Via the 2 dimensional computation  we obtain that $b_{22} = b_{12} = b_{21} = 0$.
We use a linear map in $z_1$ and $z_2$ to also normalize the 
principal $2 \times 2$ matrix of $B$, so that the submanifold restricted to
$(z_1,z_2,w)$ is in one of the normal forms B.$\gamma$, C.0, or C.1.

Finally we need to show that all entries of $B$ other than $b_{11}$ are zero.
As we have done a linear change of coordinates in $z_1$ and $z_2$,
$B$ may not be in the form \eqref{eq:initialformofAB}, but
we know $b_{jk} = 0$ as long as $j > 2$ and $k > 2$.

Now fix $k = 3,\ldots,n$.  Restrict to the submanifold given by
$z_1 = \lambda z_2$ for $\lambda = 1$ or $\lambda = -1$, and $z_j = 0$ for all
$j=3,\ldots,n$ except for $j=k$.  In the variables $(z_2,z_k,w)$,
we obtain a Levi-flat submanifold where the
matrix corresponding to $A$ is
$\left[ \begin{smallmatrix}
\lambda & 0 \\
0 & 0
\end{smallmatrix} \right]$.  The matrix corresponding to $B$ is
\begin{equation}
\begin{bmatrix}
b_{11} & b_{1k} + \lambda b_{2k} \\
b_{1k} + \lambda b_{2k} &  0
\end{bmatrix} .
\end{equation}
Via the 2 dimensional calculation we have
$b_{1k} + \lambda b_{2k} = 0$.  As this is true for $\lambda = 1$ and
$\lambda = -1$, we get that $b_{1k} = b_{2k} = 0$.

We have proved the following classification result.  It is not difficult  to
see that the submanifolds in the list are biholomorphically inequivalent
by Proposition~\ref{prop:astensors}.  The ranks of $A$ and $B$ are
invariants.  It is obvious that the $A$ matrix of B.$\gamma$
and C.x submanifolds are inequivalent.  Therefore, it is only necessary
to directly check that B.$\gamma$ are inequivalent for
different $\gamma \geq 0$, which is easy.

\begin{lemma}
If $M$ defined in $(z,w) \in \C^{n} \times \C$, $n \geq 1$, by
\begin{equation}
w = A(z,\bar{z}) + B(\bar{z},\bar{z})
\end{equation}
is Levi-flat, then $M$ is biholomorphic to one and exactly one of the
following:
\begin{equation}
\begin{aligned}
\text{(A.1)} \quad & w = \bar{z}_1^2 , \\
\text{(A.2)} \quad & w = \bar{z}_1^2 + \bar{z}_2^2 , \\
& \vdots \\
\text{(A.$n$)} \quad & w = \bar{z}_1^2 + \bar{z}_2^2 + \dots + \bar{z}_{n}^2  , \\[10pt]
\text{(B.$\gamma$)} \quad & w = \abs{z_1}^2 +  \gamma\bar{z}_1^2 , ~~ \gamma
\geq 0 , \\[10pt]
\text{(C.0)} \quad & w = \bar{z}_1z_2 , \\
\text{(C.1)} \quad & w = \bar{z}_1z_2 + \bar{z}_1^2 .
\end{aligned}
\end{equation}
\end{lemma}

The normalizing transformation used above is linear.

\begin{lemma}
If $M$ defined by
\begin{equation}
w = A(z,\bar{z}) + B(\bar{z},\bar{z}) + O(3)
\end{equation}
is Levi-flat at all points where $M$ is CR, then the quadric
\begin{equation}
w = A(z,\bar{z}) + B(\bar{z},\bar{z})
\end{equation}
is also Levi-flat.
\end{lemma}

\begin{proof}
Write $M$ as
\begin{equation}
w = A(z,\bar{z}) + B(\bar{z},\bar{z}) + r(z,\bar{z}) ,
\end{equation}
where $r$ is $O(3)$.

Let $A$ be the matrix giving the quadratic form $A(z,\bar{z})$
as before.
The Levi-map is given by taking the $n\times n$ matrix
\begin{equation}
L = L(p) =
A
+
\begin{bmatrix}
\frac{\partial^2 r}{ \partial z_j \partial \bar{z}_k}
\end{bmatrix}_{j,k}
\end{equation}
and applying it to vectors in $\pi(T^{(1,0)} M)$, where
$\pi$ is the projection onto the $\{ w = 0 \}$ plane.
That is we parametrize $M$ by the $\{ w = 0 \}$ plane,
and work there as before.

Let
\begin{equation}
\begin{aligned}
& a_j = - \overline{A}_{z_j} - \overline{B}_{z_j} - \bar{r}_{z_j} , \\
& b = \overline{A}_{z_1} + \overline{B}_{z_1} + \bar{r}_{z_1} , \\
& c = a_j (A_{z_1} + B_{z_1} + r_{z_1}) + b (A_{z_j} + B_{z_j} + r_{z_j}) .
\end{aligned}
\end{equation}
Then for $j=2,\ldots,n$,
we write the $T^{(1,0)}$ vector fields as
\begin{equation}
X_j =
a_j
\frac{\partial}{\partial z_1}
+
b
\frac{\partial}{\partial z_j}
+
c
\frac{\partial}{\partial w} .
\end{equation}
Hence $a_j \frac{\partial}{\partial z_1} + b \frac{\partial}{\partial z_j}$
are the vector fields in $\pi(T^{(1,0)} M)$.

Notice that $a_j$, $b$, and $c$ vanish at the origin, and furthermore
that if we take the linear terms of $a_j$, $b$, and the quadratic terms
in $c$, that is
\begin{equation}
\begin{aligned}
& \widetilde{a}_j = - \overline{A}_{z_j} - \overline{B}_{z_j} , \\
& \widetilde{b} = \overline{A}_{z_1} + \overline{B}_{z_1} , \\
& \widetilde{c} = \widetilde{a}_j (A_{z_1} + B_{z_1}) + \widetilde{b} (A_{z_j} + B_{z_j}) ,
\end{aligned}
\end{equation}
then away from the CR singular set of the quadric
\begin{equation}
\widetilde{X}_j =
\widetilde{a}_j
\frac{\partial}{\partial z_1}
+
\widetilde{b}
\frac{\partial}{\partial z_j}
+
\widetilde{c}
\frac{\partial}{\partial w}
\end{equation}
span the $T^{(1,0)}$ vector fields on the quadric $w = A(z,\bar{z}) +
B(\bar{z},\bar{z})$.

Since $M$ is Levi-flat, then we have that
\begin{equation}
\pi_*(X_j)^* ~L~ \pi_*(X_j) = 0 .
\end{equation}
The terms linear in $z$ and $\bar{z}$ respectively in
the expression $\pi_*(X_j)^* ~L~ \pi_*(X_j)$
are precisely
\begin{equation}
\pi_*(\widetilde{X}_j)^* ~A~ \pi_*(\widetilde{X}_j) .
\end{equation}
As this expression is identically zero,
the quadric $w = A(z,\bar{z}) + B(\bar{z},\bar{z})$ is
Levi-flat.
\end{proof}


\section{Quadratic Levi-flat submanifolds and their Segre varieties}
\label{sec:quadsegre}


A very useful invariant in CR geometry is the Segre variety.  Suppose that
a real-analytic variety $X \subset \C^N$ is defined by
\begin{equation}
\rho(z,\bar{z}) = 0 ,
\end{equation}
where $\rho$ is a real-analytic real vector-valued with $p \in X$.
Suppose that $\rho$ converges on some polydisc $\Delta$ centered at
$p$.  We 
complexify and treat $z$ and $\bar{z}$ as independent variables, and
the power series of $\rho$ at $(p,\bar{p})$ converges on $\Delta \times
\Delta$.
The Segre variety at $p$ is then defined as the variety
\begin{equation}
Q_p = \{ z \in \Delta : \rho(z,\bar{p}) = 0 \} .
\end{equation}
Of course the
variety depends on the defining equation itself and the polydisc $\Delta$.
For $\rho$ it is useful to take the defining equation or equations that generate
the ideal of the complexified $X$ in $\C^N \times \C^N$ at $p$.  If $\rho$
is polynomial we take $\Delta = \C^N$.

It is well-known that any
irreducible
complex variety that lies in $X$ and goes through
the point $p$ also lies in $Q_p$.  In case of Levi-flat submanifolds
we generally get equality as germs.  For example, for the CR Levi-flat
submanifold $M$ given by
\begin{equation}
\Im z_1 = 0, \qquad \Im z_2 = 0 ,
\end{equation}
the Segre variety $Q_0$ through the origin is precisely $\{ z_1 = z_2 =
0\}$, which happens to be the unique complex variety in $M$ through the origin.

Let us take the Levi-flat quadric
\begin{equation}
w = A(z,\bar{z}) + B(\bar{z},\bar{z}) .
\end{equation}
As we want to take the generating equations in the complexified space we
also need the conjugate
\begin{equation}
\bar{w} = \bar{A}(\bar{z},z) + \bar{B}(z,z) .
\end{equation}
The Segre variety is
then given by
\begin{equation}
w = 0, \qquad \bar{B}(z,z) = 0 .
\end{equation}

Through any CR singular point of a
real-analytic Levi-flat $M$ there is a complex variety of dimension $n-1$
that is the limit of the leaves of the Levi-foliation of $M_{CR}$,
via Lemma~\ref{lemma:dflimit}.
Let us take all possible such limits, and call their union $Q'_p$.
Notice that there could be other complex varieties in $M$ through $p$
of dimension $n-1$.  Note that $Q'_p \subset Q_p$.

Let us write down and classify the Segre varieties for all the quadric Levi-flat
submanifolds in $\C^{n+1}$:

\medskip

\begin{center}
\begin{tabular}{l|l|l|l|l|l}
Type &
Segre variety $Q_0$ &
$Q_0$ singular? &
$\dim_\C Q_0$ & $Q_0 \subset M$? &
$Q'_0$ \\
\hline
A.1 &
$w = 0$, $z_1^2 = 0$ &
no &
$n-1$ &
yes &
$Q_0$
\\
A.$k$ &
$w = 0$, $z_1^2 + \cdots + z_k^2 = 0$ &
yes &
$n-1$ &
yes &
$Q_0$
\\
B.0 &
$w = 0$ &
no &
$n$ &
no &
$w=0$, $z_1 = 0$
\\
B.$\gamma$, $\gamma > 0$ &
$w = 0$, $z_1^2 = 0$ &
no &
$n-1$ &
yes &
$Q_0$
\\
C.0 &
$w = 0$ &
no &
$n$ &
no &
$w=0$, $z_1 = 0$
\\
C.1 &
$w = 0$, $z_1^2 = 0$ &
no &
$n-1$ &
yes &
$Q_0$
\end{tabular}
\end{center}

\medskip

The submanifold C.0 also contains the complex variety $\{ w = 0, z_2 = 0 \}$,
but this variety is transversal to the leaves of the foliation, and so
cannot be in $Q'_0$

Notice that in the cases A.$k$ for all $k$, B.$\gamma$ for $\gamma > 0$, and
C.1, the variety $Q_0$ actually gives the complex variety $Q'_0$
contained in $M$ through the origin.  In these cases, the variety is
nonsingular only in the set theoretic sense.  Scheme-theoretically the
variety is always at least a double line or double hyperplane in general.


\section{The CR singularity of Levi-flats quadrics}
\label{sec:sizeofsing}

Let us study the set of CR singularities for Levi-flat quadrics.
The following proposition is well-known.

\begin{prop} \label{prop:crsingset}
Let $M \subset \C^{n+1}$ be given by
\begin{equation}
w = \rho(z,\bar{z})
\end{equation}
where $\rho$ is $O(2)$, and $M$ is not a complex submanifold.
Then the set $S$ of CR singularities of $M$
is given by
\begin{equation}
S = \{ (z,w) : \bar{\partial} \rho = 0, w = \rho(z,\bar{z}) \} .
\end{equation}
\end{prop}

\begin{proof}
In codimension 2, a real submanifold is either CR singular, complex, or
generic.
A submanifold is generic if
$\bar{\partial}$ of all the defining equations are pointwise linearly independent
(see \cite{BER:book}).
As $M$ is not complex, to find the set of CR singularities,
we find the set of points where
$M$ is not generic.  We need both defining equations for $M$,
\begin{equation}
w = \rho(z,\bar{z}), \qquad \text{and} \qquad \bar{w} = \rho(z,\bar{z}) .
\end{equation}
As
the second equation always produces a $d\bar{w}$ while the first
does not, the only way that the two can be linearly dependent is for the
$\bar{\partial}$ of the first equation to be zero.  In other words
$\bar{\partial} \rho = 0$.
\end{proof}

Let us compute and classify the CR singular sets for the 
CR singular Levi-flat quadrics.

\medskip

\begin{center}
\begin{tabular}{l|l|l|l}
Type &
CR singularity $S$ &
$\dim_\R S$ &
CR structure of $S$ \\
\hline
A.$k$ &
$z_1 = 0$, \ldots, $z_k = 0$, $w = 0$ &
$2n-2k$ &
complex
\\
B.0 &
$z_1 = 0$, $w = 0$ &
$2n-2$ &
complex
\\
B.$\frac{1}{2}$ &
$z_1 + \bar{z}_1 = 0$, $w = 0$
&
$2n-1$ &
Levi-flat
\\
B.$\gamma$, $\gamma > 0$, $\gamma \not= \frac{1}{2}$ &
$z_1 = 0$, $w=0$ &
$2n-2$ &
complex
\\
C.0 &
$z_2 = 0$, $w = 0$ &
$2n-2$ &
complex
\\
C.1 &
$z_2+2\bar{z}_1 = 0$,
$w = \frac{-z_2^2}{4}$ &
$2n-2$ &
Levi-flat
\end{tabular}
\end{center}

\medskip

By Levi-flat we mean that $S$ is a Levi-flat CR submanifold in $\{ w = 0 \}$.
There is a conjecture that a real subvariety that is Levi-flat at
CR points has a stratification by Levi-flat CR submanifolds.
This computation gives further evidence of this conjecture.


\section{Levi-foliations and images of generic Levi-flats}
\label{sec:foliation}

A CR Levi-flat submanifold $M \subset \C^n$
of codimension 2 has a certain canonical
foliation defined on it with complex analytic leaves of real codimension 2
in $M$.
The submanifold $M$ is locally
equivalent to $\R^2 \times \C^{n-2}$, defined by
\begin{equation}
\Im z_1 = 0, \qquad \Im z_2 = 0 .
\end{equation}
The leaves of the foliation are the submanifolds given by fixing
$z_1$ and $z_2$ at a real constant.
By foliation we always mean the standard nonsingular foliation as
locally comes up in the implicit function theorem.
This foliation on $M$ is called the
\emph{Levi-foliation}.  It is obvious that the Levi-foliation on $M$
extends to a neighbourhood of $M$ as a nonsingular holomorphic foliation.
The same is not true in general for CR singular submanifolds. 
We say that a smooth holomorphic
foliation $\sL$ defined in a neighborhood of $M$
is an \emph{extension} of the Levi-foliation of $M_{CR}$,
if $\sL$ and the Levi-foliation have the same germs of leaves at each CR point of $M$.
We also say that a smooth real-analytic foliation $\widetilde{\sL}$
on $M$ is an extension of the Levi-foliation on $M_{CR}$
if $\widetilde{\sL}$ and the Levi-foliation have the same germs of leaves at each CR point of $M$.
In our situation (real-analytic),
$M_{CR}$ is a dense and open subset of $M$.
This implies that
the leaves of $\sL$ and $\widetilde{\sL}$ through a CR singular point
are complex analytic submanifolds contained in $M$.
The latter could lead to an obvious obstruction to extension.
First let us see
what happens if the foliation of $M_{CR}$
is the restriction of a nonsingular holomorphic foliation of a whole
neighbourhood of $M$.

The Bishop-like quadrics, that is
A.1 and B.$\gamma$ in $\C^{n+1}$, have a Levi-foliation that extends as
a holomorphic foliation to all of $\C^{n+1}$.
That is because these submanifolds are of the form
\begin{equation} \label{eq:ntimescnm1}
N \times \C^{n-1} .
\end{equation}
For submanifolds of the form \eqref{eq:ntimescnm1} we can find normal
forms using the well-developed
theory of Bishop surfaces in $\C^2$.

\begin{prop} \label{prop:flatfolmflds}
Suppose $M \subset \C^{n+1}$ is a 
real-analytic Levi-flat CR singular submanifold
where the Levi-foliation on $M_{CR}$ extends near $p \in M$ to a
nonsingular holomorphic foliation of a neighbourhood of $p$ in $\C^{n+1}$.
Then at $p$, $M$ is locally biholomorphically equivalent to a submanifold of the form
\begin{equation} \label{eq:ntimesbishop}
N \times \C^{n-1}
\end{equation}
where $N \subset \C^2$ is a CR singular submanifold of real dimension 2.
Therefore if $M$ has a nondegenerate complex tangent, then it is
Bishop-like, that is of type A.1 or B.$\gamma$.

Furthermore, two submanifolds of the form \eqref{eq:ntimesbishop} are
locally biholomorphically (resp.\ formally) equivalent if and only if the corresponding $N$s are
locally biholomorphically  (resp.\ formally) equivalent in $\C^2$.
\end{prop}

\begin{proof}
We flatten the holomorphic foliation near $p$ so that
in some polydisc $\Delta$,
the leaves of the foliation
are given by $\{ q \} \times \C^{n-1} \cap \Delta$ for $q \in \C^2$.  Let
us suppose that $M$ is closed in $\Delta$.
At any
CR point of $M$, the leaf of the Levi-foliation agrees with
the leaf of the holomorphic foliation and therefore the leaf
that lies in $M$ agrees with a leaf of the form
$\{ q \} \times \C^{n-1}$ as a germ and so
$\{ q \} \times \C^{n-1} \cap \Delta \subset M$.
As $M_{CR}$ is dense in $M$, then $M$ is a union of
sets of the form $\{ q \} \times \C^{n-1} \cap \Delta$ and the first part
follows.

It is classical that
every Bishop surface (2 dimensional real
submanifold of $\C^2$ with a nondegenerate complex tangent)
is equivalent
to a submanifold whose quadratic part is of the form A.1 or B.$\gamma$.

Finally, the proof that two submanifolds of the form \eqref{eq:ntimesbishop} are
equivalent if and only if the $N$s are equivalent is straightforward.
\end{proof}


Not every Bishop-like submanifold is a cross product
as above.  In fact the Bishop invariant may well change from point to point.
See \S~\ref{sec:bishopexamples}.
In such cases the foliation does not extend to a nonsingular
holomorphic foliation of a neighbourhood.

Let us now focus on extending the Levi-foliation to $M$, and not to
a neighbourhood of $M$.
Let us prove a useful proposition about recognizing certain
CR singular Levi-flats from the form of the defining equation.  That is
if the $r$ in the equation does not depend on $\bar{z}_2$ through
$\bar{z}_n$.


\begin{prop} \label{prop:imagelf}
Suppose near the origin $M \subset \C^{n+1}$ is given by
\begin{equation}
w = r(z_1,\bar{z}_1,z_2,z_3,\ldots,z_n) ,
\end{equation}
where $r$ is $O(2)$ and $\frac{\partial r}{\bar{z}_1} \not\equiv 0$.
Then $M$ is a CR singular Levi-flat submanifold and the Levi-foliation
of $M_{CR}$ extends through the origin to a real-analytic
foliation on $M$.  Furthermore, there
exists a real-analytic CR mapping $F \colon U \subset \R^2 \times \C^{n-1}
\to \C^{n+1}$, $F(0) = 0$, which is a diffeomorphism onto its image $F(U)
\subset M$.
\end{prop}


Near 0, $M$ is the image of a CR mapping that is a diffeomorphism
onto its image of the standard CR Levi-flat.  The proposition 
also holds in two dimensions ($n=1$), although in this case it is somewhat
trivial.

\begin{proof}
As in \cite{LMSSZ},
let us define the mapping $F$ by
\begin{equation}
(x,y,\xi) \mapsto
\bigl(x+iy, \quad \xi, \quad r(x+iy,x-iy, \xi ) \bigr) ,
\end{equation}
where  $\xi = (\xi_2,\ldots,\xi_n) \in \C^{n-1}$.
Near points where $M$ is CR, this mapping is a CR diffeomorphism and hence
$M$ must be Levi-flat.  Furthermore, since $F$ is a diffeomorphism, it takes
the Levi-foliation on $\R^2 \times \C^{n-1}$ to a foliation on $M$ near 0.
\end{proof}

In fact, we make the following conclusion.

\begin{lemma} \label{lemma:folisabs}
Let $M \subset \C^{n+1}$ be a CR singular real-analytic Levi-flat
submanifold of codimension 2 through the origin.

Then $M$ is a CR singular Levi-flat submanifold whose Levi-foliation
of $M_{CR}$ extends through the origin to a nonsingular real-analytic
foliation on $M$ if and only if there
exists a real-analytic CR mapping $F \colon U \subset \R^2 \times \C^{n-1}
\to \C^{n+1}$, $F(0) = 0$, which is a diffeomorphism onto its image $F(U)
\subset M$.
\end{lemma}


\begin{proof}
One direction is easy and was used above.  For the other direction, suppose
that we have a foliation extending the Levi-foliation through the origin.
Let us consider $M_{CR}$ an abstract CR manifold.  That is
a manifold $M_{CR}$ together with the bundle $T^{(0,1)} M_{CR}
\subset \C \otimes T M_{CR}$.  The extended
foliation on $M$ gives a real-analytic subbundle
$\sW \subset T M$.  Since we are extending the Levi-foliation, 
when $p \in M_{CR}$, then $\sW_p = T_p^c M$, where
$T_p^c M = J(T_p^c M)$ is the complex tangent space and
$J$ is the complex structure on $\C^{n+1}$.  Since $M_{CR}$ is
dense in $M$, then $J\sW=\sW$ on $M$.

Define the real-analytic subbundle $\sV \subset \C \otimes T M$ as
\begin{equation}
\sV_p = \{ X + iJ(X) : X \in \sW_p \} .
\end{equation}
At CR points $\sV_p = T_p^{(0,1)} M$ (see for example~\cite{BER:book} page
8).
Then we can find vector fields $X^1,\ldots,X^{n-1}$ in $\sW$
such that
\begin{equation}
X^1,J(X^1),
X^2,J(X^2),\ldots,X^{n-1},J(X^{n-1})
\end{equation}
is a basis of $\sW$ near the origin.  Then the basis for $\sV$ is
given by
\begin{equation}
X^1+iJ(X^1),
X^2+iJ(X^2),\ldots,X^{n-1}+iJ(X^{n-1}).
\end{equation}
As the subbundle is integrable, we obtain that $(M,\sV)$
gives an abstract CR
manifold, which at CR points agrees with $M_{CR}$.
This manifold is Levi-flat as it is Levi-flat on a dense open set.
As it is real-analytic it
is embeddable and hence there exists a real-analytic CR diffeomorphism
from a neighbourhood of $\R^2 \times \C^{n-1}$ to a neighbourhood of 0
in $M$ (as an abstract CR manifold).  This is our mapping $F$.
\end{proof}


The quadrics
A.$k$, $k \geq 2$, defined by
\begin{equation}
w = \bar{z}_1^2 + \cdots + \bar{z}_k^2 ,
\end{equation}
contain the singular variety defined by $w = 0$, $z_1^2 + \cdots + z_k^2 =
0$, and hence the Levi-foliation cannot extend to a nonsingular
foliation of the submanifold.
The quadric A.1 does admit a holomorphic foliation, but other type A.1
submanifolds do not in general.  For example,
the submanifold 
\begin{equation}
w = \bar{z}_1^2 + \bar{z}_2^3
\end{equation}
is of type A.1 and the unique complex variety through the origin
is $0 = z_1^2 + z_2^3$, which is singular.  Therefore the foliation
cannot extend to $M$.


\section{Extending the Levi-foliation of C.x type submanifolds}

Let us prove Theorem \ref{thm:folextendsCxtype}, that is, let us start with a
type C.0 or C.1 submanifold
and show that the Levi-foliation must extend real-analytically to all of $M$.
Equivalently, let us show that 
the real analytic bundle $T^{(1,0)}M_{CR}$
extends to a real analytic subbundle of $\C \otimes TM$.  Taking real
parts we obtain an involutive subbundle of $TM$ extending
$T^cM_{CR} = \Re(T^{(1,0)} M_{CR})$.

\begin{proof}[Proof of Theorem \ref{thm:folextendsCxtype}]
Let $M$ be the submanifold given by
\begin{equation}
w = \bar{z}_1 z_2 + \epsilon \bar{z}_1^2 + r(z,\bar{z})
\end{equation}
where $\epsilon = 0,1$.  Let us treat the $z$ variables as the parameters on
$M$.
Let $\pi$ be the projection onto the $\{ w=0 \}$ plane, which is tangent to $M$
at 0 as a real $2n$-dimensional hyperplane.  We will look at all the
vectorfields on this plane $\{ w=0 \}$.
All vectors in $\pi(T^{(1,0)} M)$ can be written
in terms of $\frac{\partial}{\partial z_j}$ for $j=1,\ldots,n$.

The Levi-map is given by taking the $n\times n$ matrix
\begin{equation}
L = L(p) =
\begin{bmatrix}
0 & 1 & 0 & \cdots & 0 \\
0 & 0 & 0 & \cdots & 0 \\
0 & 0 & 0 & \cdots & 0 \\
\vdots & \vdots & \vdots & \ddots & \vdots \\
0 & 0 & 0 & \cdots & 0
\end{bmatrix}
+
\begin{bmatrix}
\frac{\partial^2 r}{ \partial z_j \partial \bar{z}_k}
\end{bmatrix}_{j,k}(p)
\end{equation}
to vectors $v \in \pi(T^{(1,0)} M)$ ($\pi$ is the projection)
as $v^* L v$.  The excess term in $L$ vanishes at 0.


Notice that for $p \in M_{CR}$, $\pi(T^{(1,0)}_p M)$ is $n-1$ dimensional.
As $M$ is Levi-flat, then $v^* L v$ vanishes for $v \in \pi(T^{(1,0)}_p M)$.
Write the vector $v = (v_1,\ldots,v_n)^t$.
The zero set of the function 
\begin{equation}
(z,v) \in \C^n \times \C^n \overset{\varphi}{\mapsto} v^* L(z,\bar{z}) v
\end{equation}
is a variety $V$ of
real codimension 2 at the origin of $\mathbf C^n\times\mathbf C^n$
because of the form of $L$.  That is, at $z=0$,
the only vectors $v$ such that $v^*Lv = 0$ are those where $v_1 = 0$ or
$v_2 = 0$.  So the codimension is at least 2.  
And we know that
$v^*Lv$ vanishes for vectors in $\pi(T^{(1,0)}_p M)$ for $p \in M$ near 0,
which is real codimension 2 at each $z$ corresponding to a CR point.
Therefore, $V \cap (\pi(M_{CR}) \times \C^n )$ has
a connected component that is equal to a connected component of the real-analytic subbundle
$\pi(T^{(1,0)} M_{CR})$. We will verify that the latter is connected.


We show below that this subbundle extends past the CR singularity.  
The key point is to show that 
the restriction of $\pi\bigl(T^{(1,0)}(M_{CR})\bigr)$ extends to a
smooth real-analytic submanifold of
$T^{(1,0)} \C^n$.
Write
\begin{equation}\label{eq:vzv}
\varphi(z,v) = v_1\bar{v}_2 + \sum a_{jk}(z) v_j\bar{v}_k
\end{equation}
where $a_{jk}(0) = 0$.  

By Proposition~\ref{prop:crsingset}, $\pi(M\setminus M_{CR})$ is contained in
\begin{equation}
z_2+2\epsilon\overline z_1+r_{\overline z_1}=0.
\end{equation}
Thus $M_{CR}$ is connected. Assume that $v\cdot\frac{\partial}{\partial z}\in T^{(1,0)}_pM$ at a CR point $p$. Then
\begin{equation}
(z_2+2\epsilon\overline z_1+r_{\overline z_1})\overline v_1+\sum_{j>1} r_{\overline z_j}\overline v_j=0.
\end{equation}
When $p$ is in the open set $U_\delta\subset\pi(M_{CR})$
defined by  $\abs{z_2+2\epsilon \overline z_1}>\abs{z}/2$ and $0<\abs{z}<\delta$, $v$ is contained in
\begin{equation}
V_C\colon \abs{v_1}\leq \abs{v}/C.
\end{equation}
When $\delta$ is
sufficiently small,  $\varphi(z,v)=0$ admits a unique solution 
\begin{equation}
v_1=f(z,v_3,\dots, v_n), \quad v_2=1
\end{equation}
by imposing $v\in V_C$. Note that 
$f$ is given by convergent power series. For $\abs{z}<\delta$, define 
\begin{equation}
w_j=\bigl(w_{j1}(z),\dots, w_{jn}(z)\bigr)\in V_C, \quad j=2,\dots, n
\end{equation}
such that $\varphi(z,w_j(z))=0$ and
\begin{equation} \label{eq:condonwj}
w_{j2}=1, \quad w_{jk}=\delta_{jk},\quad j\geq 2, k>2.
\end{equation}
To see why we can do so, fix $p\in U_\delta$.  First we can find a vector
$w_2$ in $E_p=\pi(T_p^{(1,0)}M_{CR})$ such that $v_2=1$. Otherwise,
$E_p\subset V_C$ cannnot have dimension $n-1$.  Let $E_p'$ be the vector
subspace of $E_p$ with $v_2=0$. Then $E_p'$ has rank $n-2$ and remains in
the cone $V_C$. Then $E_p'$ has an element $w_2$ with $v_2$ component being
$1$. Repeating this, we find $w_2,\dots, w_n$ in $E_p$ such that the $v_j$
component of $w_i$ is $0$ for $2<j<i$. Using linear combinations, we find a
unique basis $\{ w_2,\dots, w_n \}$ of $E_p$ that satisfies condition
\eqref{eq:condonwj}.

Assume that $C$ is sufficiently large. 
By the above uniqueness assertion on $\varphi(z,v)=0$, 
we conclude that
when $p\in U_\delta$,
 $\{w_{2}(p),\dots, w_n(p)\}$ is a base of $\pi (T^{(1,0)}_pM_{CR})$. Also it is 
 real analytic at $p=0$.  Define
\begin{equation}
\omega_j(z)=w_{j}(z)\cdot\frac{\partial}{\partial z}, \quad \abs{z}<\delta.
\end{equation}
We lift the functions $\omega_j$ via $\pi$ to a subbundle of
$\C \otimes TM$, let us call these $\widetilde{\omega}_j$.
Then consider
the vector fields
$w^*_j = 2 \Re \widetilde{\omega}_j = \widetilde{\omega}_j + \overline{\widetilde{\omega}_j}$
and $w^*_{n+j}=\Im \widetilde{\omega}_j$ for $j=2,\dots, n$.  Above
CR points over $U_\delta$,  $\tilde w_j$ is in $T M_{CR}\otimes\C$ and so tangent
to $M$.  We thus obtain a $2n-2$ dimensional real analytic subbundle of $TM$
that agrees with the real analytic real
subbundle of $TM_{CR}$ induced by the Levi-foliation  above $U_\delta$. Since $M_{CR}$ 
and the subbunldes are real analytic and $M_{CR}$ is connected, they agree  over $M_{CR}$. 

The real analytic distribution spanned by $\{\omega^*_i\}$ has constant rank ($2n-2$) everywhere and is involutive on an open subset of $M_{CR}$ and hence
everywhere.
\end{proof}

\section{CR singular set of type C.x submanifolds} \label{sec:crsingc1}

Let $M \subset \C^{n+1}$
be a codimension two Levi-flat CR singular submanifold
that is an image of $\R^2 \times \C^{n-1}$ via a real-analytic CR map, and
let $S \subset M$ be the CR singular set of $M$.
In \cite{LMSSZ} it was proved that near a generic point of $S$
exactly one of the following is true:
\begin{enumerate}[(i)]
\item $S$ is Levi-flat submanifold of dimension $2n-2$ and
CR dimension $n-2$.
\item $S$ is a complex submanifold of complex dimension $n-1$ (real dimension
$2n-2$).
\item $S$ is Levi-flat submanifold of dimension $2n-1$ and
CR dimension $n-1$.
\end{enumerate}
We only have the above
classification for a generic point of $S$,
and $S$ need not be a CR submanifold everywhere.  See \cite{LMSSZ}
for examples.

If $M$ is a Levi-flat CR singular submanifold and the Levi-foliation
of $M_{CR}$
extends to $M$, then by Lemma~\ref{lemma:folisabs} at a generic point
$S$ has to be of one of the above types.
A corollary of Theorem~\ref{thm:folextendsCxtype} is the
following result.

\begin{cor}
Suppose that $M \subset \C^{n+1}$, $n \geq 2$, is a real-analytic
Levi-flat CR singular type C.1 or type C.0 submanifold.  Let
$S \subset M$ denote the CR singular set.
Then near the origin $S$ is a submanifold of dimension $2n-2$, and
at a generic point, $S$ is either CR Levi-flat of dimension $2n-2$
(CR dimension $n-2$) or a complex submanifold of complex dimension $n-1$.

Furthermore, if $M$ is of type C.1, then
at the origin
$S$ is a CR Levi-flat submanifold of dimension $2n-2$ (CR dimension $n-2$).
\end{cor}

\begin{proof}
Let us take $M$ to be given by
\begin{equation}
w = \bar{z}_1 z_2 + \epsilon \bar{z}_1^2 + r(z,\bar{z})
\end{equation}
where $r$ is $O(3)$ and $\epsilon = 0$ or $\epsilon = 1$.

By Proposition~\ref{prop:crsingset}
the CR singular set is exactly where
\begin{equation}
z_2 + \epsilon 2 \bar{z}_1 + r_{\bar{z}_1}(z,\bar{z}) = 0, \qquad
\text{and} \qquad
r_{\bar{z}_j}(z,\bar{z}) = 0 \quad \text{for all $j=2,\ldots,n$}.
\end{equation}
By considering the real and imaginary parts of the first equation
and applying the implicit function theorem the set
$\widetilde{S} = \{ z : z_2 + \epsilon 2 \bar{z}_1 + r_{\bar{z}_1}(z,\bar{z}) =
0 \}$ is a real submanifold of real dimension $2n-2$ (real codimension 2 in
$M$).
Now $S \subset \widetilde{S}$, but as we saw above $S$ is of dimension at least
$2n-2$.
Therefore $S = \widetilde{S}$ near the origin.
The conclusion of the first part then follows from the classification above.

The stronger conclusion for C.1 submanifolds follows by noticing
that when $\epsilon = 1$, the submanifold
\begin{equation}
z_2 + 2 \bar{z}_1 + r_{\bar{z}_1}(z,\bar{z}) = 0
\end{equation}
is CR and not complex at the origin.
\end{proof}


\section{Mixed-holomorphic submanifolds} \label{sec:onebar}

Let us study sets in $\C^m$ defined by
\begin{equation} \label{eq:mixedeq}
f(\bar{z}_1,z_2,\ldots,z_m) = 0 ,
\end{equation}
for a single holomorphic function $f$ of $m$ variables.

Such sets have much in common with complex varieties, since they are
in fact complex varieties when $\bar{z}_1$ is treated as a complex variable.
The distinction is that the automorphism group is different since we
are interested in automorphisms that are holomorphic not mixed-holomorphic.

\begin{prop}
If $M \subset \C^m$ is a submanifold with a defining equation
of the form \eqref{eq:mixedeq},
where $f$ is a holomorphic function that is not identically zero,
then $M$ is a real codimension 2 set and $M$ is either a complex submanifold
or a Levi-flat submanifold, possibly CR singular.
Furthermore, if $M$ is CR singular at $p \in M$, and has a nondegenerate complex tangent
at $p$,
then $M$ has type A.$k$, C.0, or C.1 at $p$.
\end{prop} 

\begin{proof}
Since the zero set of $f$ is a complex variety in the 
$(\bar{z}_1,z_2,\ldots,z_m)$ space, we get automatically that it is real
codimension 2.  We also have that as it is a submanifold, then it can be
written as a graph of one variable over the rest.

Let $m = n+1$ for convenience and suppose that $M \subset \C^{n+1}$
is a submanifold through the origin. By factorization for a germ of holomorphic
function and by the smoothness assumption on $M$
we may assume that $df(0) \neq 0$.  Call the variables $(z_1,\ldots,z_n,w)$
and write $M$ as a graph. One possibility is that we write $M$ as
\begin{equation}
\bar{w} = \rho(z_1,\ldots,z_n),
\end{equation}
where $\rho(0) = 0$ and $\rho$ has no linear terms.
$M$ is complex if $\rho \equiv 0$.  Otherwise
$M$ is CR singular and we rewrite it as
\begin{equation}
w = \bar{\rho}(\bar{z}_1,\ldots,\bar{z}_n).
\end{equation}
We notice that the matrix representing the Levi-map must be identically
zero, so we must get Levi-flat.
If there are any quadratic terms we obtain a
type A.$k$ submanifold.

Alternatively $M$ can be written as
\begin{equation}
w = \rho(\bar{z}_1,z_2,\ldots,z_n),
\end{equation}
with $\rho(0) = 0$. If $\rho$ does not depend on $\bar{z}_1$ then $M$ is complex. Assume that $\rho$
depends on $\bar{z}_1$.
If $\rho$ has linear terms in $\bar{z}_1$, then $M$ is CR.
Otherwise it is a CR singular submanifold, and near non-CR singular points
it is a generic codimension 2 submanifold.
The CR singular set of $M$
is defined by $\frac{\partial\rho}{\partial \overline z_1}=0$. 

Suppose that $M$ is CR singular.
That $M$ is Levi-flat follows from Proposition~\ref{prop:imagelf}.
We can therefore normalize the quadratic term, after linear terms
in $z_2,\ldots,z_n$ are absorbed into $w$.

If not all quadratic terms are zero, then we notice that we
must have an A.$k$, C.0, or C.1 type submanifold.
\end{proof}

Let us now study normal forms for such sets in $\C^2$ and $\C^m$, $m \geq 3$.
First
in two variables we can easily completely answer the question.
This result is surely well-known and classical.

\begin{prop}
If $M \subset \C^2$ is a submanifold with a defining equation of the form
\eqref{eq:mixedeq}, then  it is locally
biholomorphically equivalent to a submanifold in
coordinates $(z,w) \in \C^2$ of the form
\begin{equation}
w = \bar{z}^d
\end{equation}
for $d=0,1,2,3,\ldots$ where $d$ is a local biholomorphic invariant of $M$.
If $d=0$, $M$ is complex, if $d=1$ it is a CR totally-real submanifold,
and if $d \geq 2$ then $M$ is CR singular.
\end{prop}

\begin{proof}
Write the submanifold as a graph of one
variable over the other.  Without loss of generality and after possibly
taking a conjugate of the equation, we have
\begin{equation}
w = f(\bar{z})
\end{equation}
for some holomorphic function $f$.  Assume $f(0) = 0$.
If $f$ is identically zero, then $d=0$ and we are finished.
If $f$ is not identically zero, then
it is locally biholomorphic to a positive power of the variable.
We apply a holomorphic change of coordinates in $z$, and the rest follows easily.
\end{proof}

In three or more variables, if $M \subset \C^{n+1}$, $n \geq 2$,
is a submanifold through the origin, then
if the quadratic part is nonzero
we have seen above that it can be a type A.$k$, C.0, or C.1 submanifold.
If the submanifold is the nondegenerate type C.1
submanifold, then we will show in the next section 
that $M$ is biholomorphically equivalent to the quadric $M_{C.1}$.

Before we move to C.1, let us quickly consider the mixed-holomorphic
submanifolds of type A.$n$.
The submanifolds of type A.$n$ in $\C^{n+1}$ can in some sense be considered
nondegenerate when talking about mixed-holomorphic submanifolds.

\begin{prop}
If $M \subset \C^{n+1}$ is a submanifold of type
A.$n$ at the origin of the form
\begin{equation}
w = \bar{z}_1^2+ \cdots + \bar{z}_n^2 + r(\bar{z})
\end{equation}
where $r \in O(3)$.  Then $M$ is locally near the origin biholomorphically
equivalent to the A.$n$ quadric
\begin{equation}
w = \bar{z}_1^2+ \cdots + \bar{z}_n^2 .
\end{equation}
\end{prop}

\begin{proof}
The complex Morse lemma (see e.g.\ Proposition 3.15 in \cite{Ebeling}) states that there is a local change of coordinates
near the origin in just the $z$ variables such that
\begin{equation}
z_1^2+ \cdots + z_n^2 + \bar{r}(z)
\end{equation}
is equivalent to $z_1^2+ \cdots + z_n^2$.
\end{proof}

It is not difficult to see that the normal form for mixed-holomorphic
submanifolds in $\C^{n+1}$ of type A.$k$, $k < n$,
is equivalent to a local normal form for a holomorphic function in $n$
variables.  Therefore for example the submanifold $w = \bar{z}_1^2 +
\bar{z}_2^3$ is of type A.1 and is not equivalent to any quadric.


\section{Formal normal form for certain C.1 type submanifolds I}

In this section we prove the formal normal form in
Theorem~\ref{thm:themixedholform}.  That is, we prove that
if $M \subset \C^{n+1}$
is defined by
\begin{equation}
w = \bar{z}_1 z_2 + \bar{z}_1^2 + r(z_1,\bar{z}_1,z_2,z_3,\ldots,z_n) ,
\end{equation}
where $r$ is $O(3)$, then $M$ is Levi-flat and formally equivalent to
\begin{equation}
w = \bar{z}_1z_2 + \bar{z}_1^2 .
\end{equation}

That $M$ is Levi-flat follows from Proposition~\ref{prop:imagelf}.

\begin{lemma} \label{lemma:formalpartofnormform}
If $M \subset \C^{n+1}$, $n \geq 2$, is given by
\begin{equation}
w = \bar{z}_1 z_2 + \bar{z}_1^2 + r(z_1,\bar{z}_1,z_2,z_3,\ldots,z_n)
\end{equation}
where $r$ is $O(3)$ formal power series
then $M$ is formally equivalent to $M_{C.1}$ given by
\begin{equation}
w = \bar{z}_1z_2 + \bar{z}_1^2 .
\end{equation}
In fact, the normalizing transformation can be of the form
\begin{equation}
(z,w) = (z_1,\ldots,z_n,w)
\mapsto
\bigl(
z_1, \quad f(z,w), \quad z_3, \quad \ldots, \quad z_n, \quad g(z,w)
\bigr) ,
\end{equation}
where $f$ and $g$ are formal power series.
\end{lemma}

\begin{proof}
Suppose that the normalization was done to degree $d-1$, then suppose that
\begin{equation}
w = \bar{z}_1 z_2 + \bar{z}_1^2 + r_1(z_1,\bar{z}_1,z_2,\ldots,z_n) +
r_2(z_1,\bar{z}_1,z_2,\ldots,z_n) ,
\end{equation}
where $r_1$ is degree $d$ homogeneous and $r_2$ is $O(d+1)$.
Write
\begin{equation}
r_1(z_1,\bar{z}_1,z_2,\ldots,z_n) =
\sum_{j=0}^k \sum_{\abs{\alpha}+j = d} c_{j,\alpha}
\bar{z}_1^j z^\alpha ,
\end{equation}
where $k$ is the highest power of $\bar{z}_1$ in $r_1$, and $\alpha$ is
a multiindex.

If $k$ is even, then use the transformation that replaces
$w$ with
\begin{equation}
w + \sum_{\abs{\alpha}+k = d} c_{j,\alpha} w^{k/2} z^\alpha .
\end{equation}
Let us look at the degree $d$ terms in
\begin{equation}
(\bar{z}_1 z_2 + \bar{z}_1^2)
+ \sum_{\abs{\alpha}+k = d} c_{j,\alpha}
{(\bar{z}_1 z_2 + \bar{z}_1^2)}^{k/2} z^\alpha
=
\bar{z}_1 z_2 + \bar{z}_1^2 + r_1(z_1,\bar{z}_1,z_2,\ldots,z_n) .
\end{equation}
We need not include $r_2$ as the terms are all degree $d+1$ or more.
After cancelling out the new terms on the left,
we notice that the formal transformation removed all the terms in $r_1$
with a power $\bar{z}_1^k$ and replaced them with terms that have
a smaller power of $\bar{z}_1$.

Next suppose that $k$ is odd.
We use the transformation that replaces $z_2$ with
\begin{equation}
z_2 - \sum_{\abs{\alpha}+k = d} c_{j,\alpha} w^{(k-1)/2} z^\alpha .
\end{equation}
Let us look at the degree $d$ terms in
\begin{equation}
\bar{z}_1 z_2 + \bar{z}_1^2
=
\bar{z}_1 \left(
z_2 - \sum_{\abs{\alpha}+k = d} c_{j,\alpha} w^{(k-1)/2} z^\alpha
\right)
 + \bar{z}_1^2 + r_1(z_1,\bar{z}_1,z_2,\ldots,z_n) .
\end{equation}
Again we need not include $r_2$ as the terms are all degree $d+1$ or more,
and we need not add the new terms to $z_2$ in the argument list for $r_1$
since all those terms would be of higher degree.
Again we notice that the formal transformation removed all the terms in $r_1$
with a power $\bar{z}_1^k$ and replaced them with terms that have
a smaller power of $\bar{z}_1$.

The procedure above does not change the form of the submanifold, but
it lowers the degree of $\bar{z}_1$ by one.
Since we can assume that all terms in $r_1$ depend on $\bar{z}_1$, we are
finished with degree $d$ terms after $k$ iterations of the above procedure.
\end{proof}



\section{Convergence of normalization for certain C.1 type submanifolds}

A key point in the computation below is the following
natural involution for the quadric $M_{C.1}$.  Notice
that the map 
\begin{equation}
(z_1,z_2,\ldots,z_n,w) \mapsto (-\bar{z}_2-z_1, \quad z_2, \quad \ldots,
\quad z_n, \quad w) 
\end{equation}
takes $M_{C.1}$ to itself.  The involution simply replaces
the $\bar{z}_1$ in the equation with $-z_2-\bar{z}_1$.
The way this involution is defined is by noticing that
the equation $w = \bar{z}_1 z_2 + \bar{z}_1^2$ has generically
two solutions
for $\bar{z}_1$ keeping $z_2$ and $w$ fixed.
In the same way we could define an involution on all
type C.1 submanifolds of the form 
$w = \bar{z}_1 z_2 + \bar{z}_1^2 + r(\bar{z}_1,z_2,\ldots,z_n)$,
although we will not require this construction.

We prove convergence via the following well-known lemma:

\begin{lemma} \label{lemma:convlemma}
Let $m_1, \ldots, m_N$ be positive integers.
Suppose $T(z)$ is a formal power series in $z \in \C^N$.  Suppose
$T(t^{m_1}v_1,\ldots, t^{m_N}v_N)$ is a convergent power series in $t \in \C$ for
all $v \in \C^N$.  Then $T$ is convergent.
\end{lemma}

The proof is a standard application of the Baire category theorem and the Cauchy inequality.
See \cite{BER:book} (Theorem 5.5.30, p.\ 153) where all $m_j$ are $1$.
For $m_j > 1$ we first change variables by setting $v_j = w_j^{m_j}$ and
apply the lemma with $m_j=1$.

The following lemma finishes the proof of
Theorem~\ref{thm:themixedholform}.  
By absorbing any holomorphic terms into $w$,
we assume that $r(z_1,0,z_2,\ldots,z_n) \equiv 0$.
In Lemma~\ref{lemma:formalpartofnormform}
we have
also constructed a formal transformation that only changed the $z_2$
and $w$ coordinates, so it is enough to prove convergence in this case.
Key points of this proof are that the right hand
side of the defining equation for $M_{C.1}$ is homogeneous, and that
we have a natural involution on $M_{C.1}$.

\begin{lemma} \label{lemma:convergence}
If $M \subset \C^{n+1}$, $n \geq 2$, is given by
\begin{equation}
w = \bar{z}_1 z_2 + \bar{z}_1^2 + r(z_1,\bar{z}_1,z_2,z_3,\ldots,z_n)
\end{equation}
where $r$ is $O(3)$ and convergent, and $r(z_1,0,z_2,\ldots,z_n) \equiv 0$.
Suppose that two formal power series $f(z,w)$ and $g(z,w)$ satisfy 
\begin{equation}\label{gzbz}
g(z,\bar{z}_1z_2 + \bar{z}_1^2) = \bar{z}_1 f(z,\bar{z}_1z_2 + \bar{z}_1^2)
+ \bar{z}_1^2 + r(z_1,\bar{z}_1,f(z,\bar{z}_1z_2 +
\bar{z}_1^2),z_3,\ldots,z_n) .
\end{equation}
Then $f$ and $g$ are convergent.
\end{lemma}

\begin{proof}
The equation \eqref{gzbz} is true formally, 
treating $z_1$ and $\bar{z}_1$
as independent variables. Notice that \eqref{gzbz} has one equation for
2 unknown functions.

We now use the involution on $M_{C.1}$ to create a system
that we can solve uniquely.  We replace
$\bar{z}_1$ with $-z_2-\bar{z}_1$.  We leave $z_1$ untouched (treating as an
independent variable).  We obtain an identity in formal power series:
\begin{multline}
g(z,\bar{z}_1z_2 + \bar{z}_1^2) = (-z_2-\bar{z}_1) f(z,\bar{z}_1z_2 + \bar{z}_1^2)
+ (-z_2-\bar{z}_1)^2 \\
+ r(z_1,(-z_2-\bar{z}_1),f(z,\bar{z}_1z_2 + \bar{z}_1^2),z_3,\ldots,z_n) .
\end{multline}
The formal series
$\xi = f(z,\bar{z}_1z_2 + \bar{z}_1^2)$ and
$\omega = g(z,\bar{z}_1z_2 + \bar{z}_1^2)$ are solutions of the system
\begin{align}
\omega & = \bar{z}_1 \xi
+ \bar{z}_1^2
+ r(z_1,\bar{z}_1,\xi,z_3,\ldots,z_n) , \\
\omega & = (-z_2-\bar{z}_1) \xi
+ (-z_2-\bar{z}_1)^2
+ r(z_1,(-z_2-\bar{z}_1),\xi,z_3,\ldots,z_n) .
\end{align}

We next replace $z_j$ with $t z_j$ and $\bar{z}_1$ with $t \bar{z}_1$
for $t \in \C$.  Because $\bar{z}_1z_2 + \bar{z}_1^2$ is homogeneous of
degree 2, we obtain that for every $(z_1,\bar{z}_1,z_2,\ldots,z_n) \in
\C^{n+1}$ the formal series in $t$ given by
$\xi(t) = f\bigl(tz,t^2(\bar{z}_1z_2 + \bar{z}_1^2)\bigr)$,
$\omega(t) = g\bigl(tz,t^2(\bar{z}_1z_2 + \bar{z}_1^2)\bigr)$
are solutions of the system
\begin{align}
\omega & = t \bar{z}_1 \xi
+ t^2 \bar{z}_1^2
+ r(tz_1,t\bar{z}_1,\xi,tz_3,\ldots,tz_n) , \label{eq:formalsolconvt} \\
\omega & = t (-z_2-\bar{z}_1) \xi
+ t^2 (-z_2-\bar{z}_1)^2
+ r(tz_1,t(-z_2-\bar{z}_1),\xi,t z_3,\ldots, t z_n) .
\end{align}
We eliminate $\omega$ to obtain an equation for $\xi$:
\begin{equation}
t (2 \bar{z}_1 + z_2) ( \xi - t z_2)
=
\\
r(tz_1,t(-z_2-\bar{z}_1),\xi,t z_3,\ldots, t z_n)
- r(tz_1,t\bar{z}_1,\xi,tz_3,\ldots,tz_n) .
\end{equation}
We now treat $\xi$ as a variable and we have a holomorphic (convergent)
equation.  The right hand size must be divisible by
$t (2 \bar{z}_1 + z_2)$:  It is divisible by $t$ since
$r$ was divisible by $\bar{z}_1$.  It is also divisible by
$2 \bar{z}_1 + z_2$ as setting $z_2 = -2 \bar{z}_1$ makes the right
hand side vanish.  Therefore,
\begin{equation}
\xi - t z_2
=
\frac{
r(tz_1,t(-z_2-\bar{z}_1),\xi,t z_3,\ldots, t z_n)
- r(tz_1,t\bar{z}_1,\xi,tz_3,\ldots,tz_n)
}{t (2 \bar{z}_1 + z_2)} ,
\end{equation}
where the right hand side is a holomorphic function (that is, a convergent
power series) in $z_1,\bar{z}_1,z_2,\ldots,z_n,t,\xi$.
For any fixed $z_1,\bar{z}_1,z_2,\ldots,z_n$, we solve for $\xi$ in terms of $t$
via the implicit function theorem,
and we obtain that $\xi$ is a holomorphic
function of $t$.  The power series of $\xi$ is given by
$\xi(t) = f\bigl(tz,t^2(\bar{z}_1z_2 + \bar{z}_1^2)\bigr)$.


Let $v \in \C^{n+1}$ be any nonzero vector.  Via a proper choice
of $z_1,\bar{z}_1,z_2,\ldots,z_n$ (still treating $\bar{z}_1$
and $z_1$ as independent variables) we write $v =
(z,\bar{z}_1z_2 + \bar{z}_1^2)$.
We apply the above argument to
$\xi(t) = f(tv_1,\ldots, tv_n, t^2v_{n+1})$, and $\xi(t)$
converges as a series in $t$.
As we get convergence for
every $v \in \C^{n+1}$ we obtain that $f$ converges
by Lemma~\ref{lemma:convlemma}.
Once $f$ converges, then
via \eqref{eq:formalsolconvt} we obtain that $g(tv_1,\ldots, tv_n, t^2v_{n+1})$
converges as
a series in $t$ for all $v$, and hence $g$ converges.
\end{proof}


\section{Automorphism group of the C.1 quadric}


With the normal form achieved 
in previous sections, let us study the
automorphism group of the C.1 quadric in this section.
We will again use the mixed-holomorphic
involution that is obtained
from the quadric. 

We study the local automorphism group at the origin.  That is
the set of 
germs at the origin of biholomorphic
transformations taking $M$ to $M$ and fixing the origin.

First we look at the linear parts of automorphisms.  We already know
that the linear term of the last component only depends on $w$.  For
$M_{C.1}$ we can say more about the first two components.

\begin{prop} \label{prop:C1linpart}
Let $(F,G) = (F_1,\ldots,F_n,G)$ be a formal invertible or biholomorphic
automorphism of $M_{C.1} \subset \C^{n+1}$, that is the submanifold of the form
\begin{equation}
w = \bar{z}_1 z_2 + \bar{z}_1^2.
\end{equation}
Then $F_1(z,w) = a z_1 + \alpha w + O(2)$, 
$F_2(z,w) = \bar{a} z_2 + \beta w + O(2)$, and $G(z,w) = \bar{a}^2 w + O(2)$,
where $a \not= 0$.
\end{prop}

\begin{proof}
Let $a = (a_1,\ldots,a_n)$ and
$b = (b_1,\ldots,b_n)$ be such that $F_1(z,w) = a \cdot z + \alpha w + O(2)$
and $F_2(z,w) = b \cdot z + \beta w + O(2)$.  Then from
Proposition~\ref{prop:astensors} we have
\begin{equation}
\begin{bmatrix}
0 & 1 & 0 & \cdots & 0 \\
0 & 0 & 0 & \cdots & 0 \\
\vdots & \vdots &\vdots & \ddots & \vdots \\
0 & 0 & 0 & \cdots & 0
\end{bmatrix}
=
\lambda
\begin{bmatrix} a^* & b^* & \cdots \end{bmatrix}
\begin{bmatrix}
0 & 1 & 0 & \cdots & 0 \\
0 & 0 & 0 & \cdots & 0 \\
\vdots & \vdots &\vdots & \ddots & \vdots \\
0 & 0 & 0 & \cdots & 0
\end{bmatrix}
\begin{bmatrix} a \\ b \\ \vdots \end{bmatrix} .
\end{equation}
Therefore $\lambda \bar{a}_1 b_2 = 1$, and $\bar{a}_j b_k = 0$ for all $(j,k) \not=
(1,2)$.  Similarly
\begin{equation}
\begin{bmatrix}
1 & 0 & 0 & \cdots & 0 \\
0 & 0 & 0 & \cdots & 0 \\
\vdots & \vdots &\vdots & \ddots & \vdots \\
0 & 0 & 0 & \cdots & 0
\end{bmatrix}
=
\lambda
\begin{bmatrix} a^* & b^* & \cdots \end{bmatrix}
\begin{bmatrix}
1 & 0 & 0 & \cdots & 0 \\
0 & 0 & 0 & \cdots & 0 \\
\vdots & \vdots &\vdots & \ddots & \vdots \\
0 & 0 & 0 & \cdots & 0
\end{bmatrix}
\begin{bmatrix} \bar{a} \\ \bar{b} \\ \vdots \end{bmatrix} .
\end{equation}
Therefore $\lambda \bar{a}_1^2 = 1$, and $\bar{a}_j \bar{a}_k = 0$ for all $(j,k)
\not= (1,1)$.  Putting these two together we obtain that $a_j = 0$ for all
$j \not= 1$, and as $a_1 \not= 0$ we get $b_j = 0$ for all $j \not= 2$.  As
$\lambda$ is the reciprocal of the coefficient of $w$ in $G$, we are
finished.
\end{proof}

\begin{lemma} \label{lemma:mixholC3bihol}
Let $M_{C.1} \subset \C^3$ be given by
\begin{equation}
w = \bar{z}_1 z_2 + \bar{z}_1^2.
\end{equation}
Suppose that a local biholomorphism (resp.\ formal automorphism) $(F_1,F_2,G)$
transforms $M_{C.1}$ into $M_{C.1}$. Then $F_1$
depends only on $z_1$, and $F_2$ and $G$ depend only on
$z_2$ and $w$.
\end{lemma}

\begin{proof}
Let us define a $(1,0)$ tangent vector field on $M$ by
\begin{equation}
Z=\frac{\partial}{\partial z_2} + \bar{z}_1 \frac{\partial}{\partial w} .
\end{equation}
Write $F = (F_1,F_2,G)$.
$F$ must take $Z$ into a multiple of itself
when
restricted to $M_{C.1}$.  That is on $M_{C.1}$ we have
\begin{align}
& \frac{\partial F_1}{\partial z_2} + \bar{z}_1 \frac{\partial F_1}{\partial w}
= 0 ,
\\
& \frac{\partial F_2}{\partial z_2} + \bar{z}_1 \frac{\partial F_2}{\partial w}
= \lambda ,
\\
& \frac{\partial G}{\partial z_2} + \bar{z}_1 \frac{\partial G}{\partial w}
= \lambda \overline{F_1}(\bar{z},\bar{w}) ,
\end{align}
for some function $\lambda$.
Let us take the first equation and plug in the defining equation for $M_1$:
\begin{equation} \label{eq:geqt10}
\frac{\partial F_1}{\partial z_2}
(z_1,z_2,
  \bar{z}_1 z_2 + \bar{z}_1^2
)
+
\bar{z}_1
\frac{\partial F_1}{\partial w}
(z_1,z_2,
  \bar{z}_1 z_2 + \bar{z}_1^2
)
= 0 .
\end{equation}
This equation is true for all $z \in \C^2$, and so we may treat
$z_1$ and $\bar{z}_1$ as independent variables.

We have an involution on $M_{C.1}$ that takes $\bar{z}_1$ to
$-z_2-\bar{z}_1$.
Therefore we also have
\begin{equation} 
\frac{\partial F_1}{\partial z_2}
(z_1,z_2,
  \bar{z}_1 z_2 + \bar{z}_1^2
)
+
(-z_2-\bar{z}_1)
\frac{\partial F_1}{\partial w}
(z_1,z_2,
  \bar{z}_1 z_2 + \bar{z}_1^2
)
= 0 .
\end{equation}
%
%
%
This means that
$\frac{\partial F_1}{\partial w}$ and therefore
$\frac{\partial F_1}{\partial z_2}$ must be identically zero.
That is, $F_1$ only depends on $z_1$.

We have that the following must hold for all $z$:
\begin{equation}
G(z_1,z_2,\bar{z}_1 z_2 + \bar{z}_1^2 )
=
\overline{F_1}(\bar{z}_1)
F_2(z_1,z_2,\bar{z}_1 z_2 + \bar{z}_1^2)
+
{\left(\overline{F_1}(\bar{z}_1) \right)}^2 .
\end{equation}
Again we treat $z_1$ and $\bar{z}_1$ as independent variables.  We
differentiate
with respect to $z_1$:
\begin{equation}
\frac{\partial G}{\partial z_1}(z_1,z_2,\bar{z}_1 z_2 + \bar{z}_1^2 )
=
\overline{F_1}(\bar{z}_1)
\frac{\partial F_2}{\partial z_1}(z_1,z_2,\bar{z}_1 z_2 + \bar{z}_1^2) .
\end{equation}
We plug in the involution again to obtain
\begin{equation}
\frac{\partial G}{\partial z_1}(z_1,z_2,\bar{z}_1 z_2 + \bar{z}_1^2 )
=
\overline{F_1}(-z_2-\bar{z}_1)
\frac{\partial F_2}{\partial z_1}(z_1,z_2,\bar{z}_1 z_2 + \bar{z}_1^2) .
\end{equation}
Therefore as $F_1$ is not identically zero, then as before
both
$\frac{\partial F_2}{\partial z_1}$ and
$\frac{\partial G}{\partial z_1}$ must be identically zero.
\end{proof}

\begin{lemma} \label{lemma:F1determF2G}
Take $M_{C.1} \subset \C^3$ given by
\begin{equation}
w = \bar{z}_1 z_2 + \bar{z}_1^2 ,
\end{equation}
and let $(F_1,F_2,G)$ be a local automorphism at the origin.
Then $F_1$ uniquely determines $F_2$ and $G$.
Furthermore, given any
invertible function of one variable $F_1$ with $F_1(0) = 0$,
there exist unique $F_2$ and $G$ that complete an automorphism
and they are determined by
\begin{equation} \label{eq:F2G}
\begin{aligned}
F_2(z_2,\bar{z}_1z_2+\bar{z}_1^2) & = \bar{F}_1(\bar{z}_1)+\bar{F}_1(-\bar{z}_1-z_2), 
\\
G(z_2,\bar{z}_1z_2+\bar{z}_1^2) & = -\bar{F}_1(\bar{z}_1)\bar{F}_1(-\bar{z}_1-z_2).
\end{aligned}
\end{equation}
\end{lemma}

We should note that the lemma also works formally.  Given any formal $F_1$,
there exist unique formal $F_2$ and $G$ satisfying the above property.

\begin{proof}
By Lemma~\ref{lemma:mixholC3bihol}, $F_1$ depends only on $z_1$ and
$F_2$ and $G$ depend only on $z_2$ and $w$.
We write the automorphism
as a composition of the two mappings $\bigl(F_1(z_1),z_2,w\bigr)$ and
$\bigl(z_1,F_2(z_2,w),G(z_2,w)\bigr)$.

We plug the transformation into the defining equation for $M_{C.1}$.
\begin{equation} \label{eq:autoformtosolve}
G(z_2,\bar{z}_1z_2 + \bar{z}_1^2) = \bar{F}_1(\bar{z}_1)F_2(z_2,\bar{z}_1z_2 + \bar{z}_1^2)
+{\bigl(\bar{F}_1(\bar{z}_1)\bigr)}^2 .
\end{equation}
We use the involution $(z_1,z_2) \mapsto (-\bar{z}_1-z_2,z_2)$ which
preserves $M_{C.1}$ and obtain a second equation
\begin{equation} \label{eq:autoformtosolve2}
G(z_2,\bar{z}_1z_2 + \bar{z}_1^2) = \bar{F}_1(-\bar{z}_1-z_2)F_2(z_2,\bar{z}_1z_2 + \bar{z}_1^2)
+{\bigl(\bar{F}_1(-\bar{z}_1-z_2)\bigr)}^2 .
\end{equation}
We eliminate $G$ and solve for $F_2$:
\begin{equation} \label{f2z2}
 F_2(z_2,\bar{z}_1z_2 + \bar{z}_1^2)
= \frac{{\bigl(\bar{F}_1(-\bar{z}_1-z_2)\bigr)}^2
-{\bigl(\bar{F}_1(\bar{z}_1)\bigr)}^2}{\bar{F}_1(\bar{z}_1)-\bar{F}_1(-\bar{z}_1-z_2)}
=
\bar{F}_1(\bar{z}_1)+ \bar{F}_1(-\bar{z}_1-z_2) .
\end{equation}
Next we note that trivially, $F_2$ is unique if it exists: its difference vanishes on
$M_{C.1}$.

If we suppose that $F_1$ is convergent, then just as before,
substituting $z_2$ with $tz_2$ and $\bar{z}_1$ with $t\bar{z}_1$,
we are restricting to curves $(tz_2,t^2w)$ for all
$(z_2,w)$.  The series is convergent in $t$ for every fixed $z_2$ and $w$.
Therefore if $F_2$ exists and $F_1$ is convergent, then
$F_2$ is convergent by Lemma~\ref{lemma:convlemma}.

Now we need to show the existence of the formal solution $F_2$.
Notice that the right-hand side
of \eqref{f2z2} is invariant under the involution. It suffices to show that any power series in $\bar{z_1}, z_2$ that is invariant under
the involution is a formal power series in $z_2$ and $\bar{z}_1z_2+\bar{z}_1^2$.  Let us treat $\xi=\bar{z}_1$ as an independent
variable. The  original involution becomes a holomorphic involution in $\xi,z_2$:
\begin{equation}
\tau\colon \xi\to-\xi-z_2, \qquad z_2\to z_2.
\end{equation}
By a theorem of Noether we obtain a set of generators for the 
ring of invariants can be obtained by applying the averaging operation
$R(f) = \frac{1}{2} ( f + f \circ \tau)$ to
all monomials in $\xi$ and $z_2$ of degree 2 or less.
By direct calculation it is not difficult to see that
$\xi,\xi z_2+\xi^2$ generate the ring of invariants.  Therefore
any invariant power series in $z_2,\xi$ is a power series in $\xi,\xi z_2+\xi^2$. This shows the existence 
of $F_2$.  The existence of $G$ follows the same.

The equation for $G(z_2,\bar{z}_1z_2+\bar{z}_1^2)=-\bar{F}_1(\bar{z}_1)\bar{F}_1(-\bar{z}_1-z_2)$ is obtained by plugging in the equation
for $F_2$.  Its existence, uniqueness, and convergence
in case $F_1$ converges, follows exactly the same as for $F_2$.
\end{proof}

\begin{thm} \label{thm:autogroup}
If $M \subset \C^{n+1}$, $n \geq 2$ is given by
\begin{equation}
w = \bar{z}_1 z_2 + \bar{z}_1^2 ,
\end{equation}
and $(F_1,F_2,\ldots,F_n,G)$ is a local automorphism at the origin,
then $F_1$ depends only
on $z_1$, $F_2$ and $G$ depend only on $z_2$ and $w$, and $F_1$ completely
determines $F_2$ and $G$ via \eqref{eq:F2G}.
The mapping $(z_1,z_2,F_3,\ldots,F_n)$ has
rank $n$ at the origin.

Furthermore, given any
invertible function $F_1$ of one variable with $F_1(0) = 0$,
and arbitrary holomorphic functions $F_3,\ldots,F_n$ with $F_j(0) = 0$, and
such that $(z_1,z_2,F_3,\ldots,F_n)$ has rank $n$ at the origin,
then there exist unique $F_2$ and $G$ that complete an automorphism.
\end{thm}

\begin{proof}
Let $(F_1,\ldots,F_n,G)$ be an automorphism.  Then we have
\begin{equation}
G(z_1,\ldots,z_n,w) =
\overline{F_1}(\bar{z}_1,\ldots,\bar{z}_n,\bar{w})
F_2(z_1,\ldots,z_n,w) +
{\bigl( \overline{F_1}(\bar{z}_1,\ldots,\bar{z}_n,\bar{w}) \bigr)}^2 .
\end{equation}
Proposition~\ref{prop:C1linpart} says that the linear terms in $G$
only depend on
$w$, the linear terms of $F_1$ depend only on $z_1$ and $w$ and the linear
terms of
$F_2$ only depend on $z_2$ and $w$.

Let us embed $M_{C.1} \subset \C^3$ into $M$ via setting
$z_3 = \alpha_3 z_2$, $\ldots$, $z_n = \alpha_n z_2$, for arbitrary
$\alpha_3,\ldots,\alpha_n$.  Then we obtain
\begin{multline} \label{eq:restrictedautC1}
G(z_1,z_2,\alpha_3z_2,\ldots,\alpha_n z_2,w) = \\
\overline{F_1}(\bar{z}_1,\bar{z}_2,\bar{\alpha}_3
\bar{z}_2,\ldots,\bar{\alpha}_n\bar{z}_2,\bar{w})
F_2(z_1,z_2,\alpha_3z_2,\ldots,\alpha_n z_2,w) + \\
{\bigl( \overline{F_1}(\bar{z}_1,\bar{z}_2,\bar{\alpha}_3
\bar{z}_2,\ldots,\bar{\alpha}_n\bar{z}_2,\bar{w}) \bigr)}^2 .
\end{multline}
By noting what the linear terms are, we notice that the above is the
equation for an automorphism of $M_{C.1}$.  Therefore by
Lemma~\ref{lemma:mixholC3bihol} we have 
\begin{equation}
\frac{\partial F_1}{\partial w} = 0 \qquad \text{and} \qquad
\frac{\partial F_2}{\partial z_1} = 0 \qquad \text{and} \qquad
\frac{\partial G}{\partial z_1} = 0 ,
\end{equation}
as that is true for all $\alpha_3,\ldots,\alpha_n$.  Plugging in
the defining equation for $M_{C.1}$ we obtain an equation that holds
for all $z$ and we can treat $z$ and $\bar{z}$ independently.  We 
plug in $z = 0$ to obtain
\begin{equation}
0 =
\overline{F_1}(\bar{z}_1,\bar{z}_2,\bar{\alpha}_3
\bar{z}_2,\ldots,\bar{\alpha}_n\bar{z}_2,0)
F_2(0,\ldots,0,\bar{z}_1^2) + \\
{\bigl( \overline{F_1}(\bar{z}_1,\bar{z}_2,\bar{\alpha}_3
\bar{z}_2,\ldots,\bar{\alpha}_n\bar{z}_2,0) \bigr)}^2 .
\end{equation}
Differentiating with respect to $\bar{\alpha}_j$ we obtain 
$\frac{\partial F_1}{\partial z_j} = 0$, for $j=3,\ldots,n$.  We 
set $\bar{\alpha}_j = 0$ in the equation, differentiate with
respect to $\bar{z}_2$ and obtain that
$\frac{\partial F_1}{\partial z_2} = 0$.  In other words $F_1$ is a
function of $z_1$ only.
We rewrite
\eqref{eq:restrictedautC1} by writing $F_1$ as a function
of $z_1$ only and $F_2$ and $G$ as functions of $z_2,\ldots,z_n,w$,
and we plug in $w = \bar{z}_1z_2 + \bar{z}_1^2$ to obtain
\begin{multline}
G(z_2,\alpha_3z_2,\ldots,\alpha_n z_2,\bar{z}_1z_2 + \bar{z}_1^2) =  \\
\overline{F_1}(\bar{z}_1)
F_2(z_2,\alpha_3z_2,\ldots,\alpha_n z_2,\bar{z}_1z_2 + \bar{z}_1^2) +
{\bigl( \overline{F_1}(\bar{z}_1) \bigr)}^2 .
\end{multline}
By Lemma~\ref{lemma:F1determF2G}, we know that $F_1$ now uniquely
determines
$F_2(z_2,\alpha_3z_2,\ldots,\alpha_n z_2,w)$ and
$G(z_2,\alpha_3z_2,\ldots,\alpha_n z_2,w)$.  These two functions therefore
do not depend on $\alpha_3,\ldots,\alpha_n$, and in turn
$F_2$ and $G$ do not depend on $z_3,\ldots,z_n$ as claimed.
Furthermore $F_1$ does uniquely determine $F_2$ and $G$.

Finally since the mapping is a biholomorphism, and from what we know about the
linear parts of $F_1$, $F_2$, and $G$, it is clear that
$(z_1,z_2,F_3,\ldots,F_n)$ is rank $n$.

The other direction follows by applying
Lemma~\ref{lemma:F1determF2G}.  We start with $F_1$, determine $F_2$ and $G$
as in 3 dimensions.  Then adding $F_3,\ldots,F_n$ and the rank condition
guarantees an automorphism.
\end{proof}


\section{Normal form for certain C.1 type submanifolds II}

The goal of this section is to find the normal form
for Levi-flat submanifolds $M \subset \C^{n+1}$ given by
\begin{equation}
w = \bar{z}_1 z_2 + \bar{z}_1^2 + \Re f(z) ,
\end{equation}
for a holomorphic $f(z)$ of order $O(3)$.

Since $f(z)$ can be absorbed into $w$ via a holomorphic transformation,
the goal is really to prove the following theorem.

\begin{thm} \label{thm:flatC1}
Let $M \subset \C^{n+1}$ be a real-analytic Levi-flat given by
\begin{equation} \label{eq:flatC1}
w = \bar{z}_1 z_2 + \bar{z}_1^2 + r(\bar{z}) ,
\end{equation}
where $r$ is $O(3)$.
Then $M$ can be put into the $M_{C.1}$ normal form
\begin{equation} \label{eq:flatC1normform1}
w = \bar{z}_1z_2 + \bar{z}_1^2 ,
\end{equation}
by a convergent normalizing transformation.

Furthermore, if $r$ is a polynomial and the coefficient of $\bar{z}_1^3$
in $r$ is zero, then there exists an invertible polynomial mapping
taking $M_{C.1}$ to $M$.
\end{thm}

In Theorem~\ref{thm:themixedholform},
we have already shown that a submanifold of the form
\begin{equation}
w = \bar{z}_1 z_2 + \bar{z}_1^2 + r(\bar{z}_1)
\end{equation}
is necessarily Levi-flat and has the normal form $M_{C.1}$.
The first part of Theorem~\ref{thm:flatC1} will follow once we prove:

\begin{lemma}
If $M \subset \C^{n+1}$ is given by
\begin{equation}
w = \bar{z}_1 z_2 + \bar{z}_1^2 + r(\bar{z})
\end{equation}
where $r$ is $O(3)$ and $M$ is Levi-flat,
then $r$ depends only on $\bar{z}_1$.
\end{lemma}

\begin{proof}
First let us assume that $n=2$.
For $p \in M_{CR}$, $T^{(1,0)}_p M$ is one dimensional.
The Levi-map is the
matrix
\begin{equation}
L =
\begin{bmatrix}
0 & 1 & 0 \\
0 & 0 & 0 \\
0 & 0 & 0
\end{bmatrix}
\end{equation}
applied to the $T^{(1,0)} M$ vectors.  As $M$ is Levi-flat,
then the Levi-map has to vanish.  The only vectors $v$
for which $v^* L v = 0$, are the ones without 
$\frac{\partial}{\partial z_1}$ component or
$\frac{\partial}{\partial z_2}$ component.
That is vectors of the form
\begin{equation}
a \frac{\partial}{\partial z_1} + b \frac{\partial}{\partial w},
\qquad \text{or} \qquad
a \frac{\partial}{\partial z_2} + b \frac{\partial}{\partial w}
.
\end{equation}
We apply these vectors to the defining equation and its conjugate and we
obtain in the first case the equations
\begin{equation}
b = 0, \qquad
a \left( \bar{z}_2 + 2z_1 + \frac{\partial \bar{r}}{\partial
z_1} \right) = 0 .
\end{equation}
This cannot be satisfied identically on $M$ since this is supposed to be
true for all $z$, but $a$ cannot be identically zero and the second factor
in the second equation has only one nonholomorphic term, which is
$\bar{z}_2$.

Let us try the second form and we obtain the equations
\begin{equation}
b = a \bar{z}_1 , \qquad
a \left( \frac{\partial \bar{r}}{\partial z_2} \right) = 0 .
\end{equation}
Again $a$ cannot be identically zero, and hence the second factor of the
second equation $\frac{\partial \bar{r}}{\partial z_2}$ must be identically
zero, which is possible only if $r$ depends only on $\bar{z}_1$.

Finally, it
is possible to pick $b=\bar{z}_1$ and $a=1$, to obtain a $T^{(1,0)}$
vector field
\begin{equation}
\frac{\partial}{\partial z_2} + \bar{z}_1 \frac{\partial}{\partial w} ,
\end{equation}
and therefore
these submanifolds are necessarily Levi-flat.

Next suppose that $n > 2$.  Notice that replacing $z_k$ with
$\lambda_k \xi$ for $k \geq 2$ and then fixing $\lambda_k$ for $k \geq 2$, we
get
\begin{equation}
w = \bar{z}_1 \lambda_2 \xi + \bar{z}_1^2 + r(\bar{z}_1,
\bar{\lambda}_2 \bar{\xi},\dots,\bar{\lambda}_n \bar{\xi}) .
\end{equation}
By Lemma~\ref{lemma:restriction},
we obtain a Levi-flat submanifold in $(z_1,\xi,w) \in \C^3$, and hence
can apply the above reasoning to obtain that
$r(\bar{z}_1,
\bar{\lambda}_2 \bar{\xi},\dots,\bar{\lambda}_n \bar{\xi})$ does not
depend on $\bar{\xi}$.  As this was true for any $\lambda_k$'s, we have
that $r$ can only depend on $\bar{z}_1$.
\end{proof}


It is left to prove the claim about the polynomial normalizing
transformation.

\begin{lemma}
Suppose $M \subset \C^{n+1}$ is given by
\begin{equation}
w = \bar{z}_1 z_2 + \bar{z}_1^2 + r(\bar{z}_1)
\end{equation}
where $r$ is a polynomial that vanishes to fourth order.  Then there
exists an invertible polynomial mapping taking $M_{C.1}$ to $M$.
\end{lemma}

\begin{proof}
We will take a transformation of the form
\begin{equation}
(z_1,z_2,w) \mapsto \bigl(z_1,z_2+f(z_2,w),w+g(z_2,w) \bigr) .
\end{equation}
We are therefore trying to find polynomial $f$ and $g$ that satisfy
\begin{equation}
\bar{z}_1z_2 + \bar{z}_1^2+g(z_2,\bar{z}_1z_2 + \bar{z}_1^2)
=
\bar{z}_1 \bigl(z_2 +f(z_2,\bar{z}_1z_2 +
\bar{z}_1^2)\bigr) + \bar{z}_1^2 + r(\bar{z}_1) .
\end{equation}
If we simplify we obtain  \begin{equation}
g(z_2,\bar{z}_1z_2 + \bar{z}_1^2)
-
\bar{z}_1 f(z_2,\bar{z}_1z_2 + \bar{z}_1^2)
= r(\bar{z}_1) .
\end{equation}
Consider the involution $S\colon (\bar{z}_1,z_2)\to (-\bar{z}_1-z_2,z_2)$.
Its invariant polynomials $u(\bar{z}_1,z_2)$
are precisely the polynomials in $z_2,z_2\bar{z}_1+\bar{z}_1^2$.  The polynomial
$r(\bar{z}_1)$ can be uniquely written as
\begin{equation}
r^+(z_2,\bar{z}_1z_2+\bar{z}_1^2)+\Bigl(\bar{z}_1+\frac{z_2}{2}\Bigr)r^-(z_2,\bar{z}_1z_2+\bar{z}_1^2)
\end{equation}
in two polynomials $r^\pm$. 
Taking $f=-r^-$ and $g=r^++\frac{z_2}{2}r^-$, we find the desired solutions. 
\end{proof}


\section{Normal form for general type C.1 submanifolds} \label{sec:generalC1}

In this section we show that generically a Levi-flat type C.1 submanifold is
not formally equivalent to the quadric $M_{C.1}$ submanifold.  In fact, we
find a formal normal form that shows infinitely many invariants.
There are obviously infinitely many invariants if we do not impose the
Levi-flat condition.  The trick therefore is, how to impose the Levi-flat
condition and still obtain a formal normal form.

Let $M \subset \C^3$ be a real-analytic
Levi-flat type C.1 submanifold through the origin.
We know that $M$ is an image of $\R^2 \times \C$
under a real-analytic CR map that is a diffeomorphism onto its
target; see Theorem~\ref{thm:folextendsCxtype}.
After a linear change of coordinates we assume that
the mapping is
\begin{equation}
\begin{split}
(x,y,\xi) \in \R^2 \times \C \mapsto
\bigl(
& x+iy + a(x,y,\xi), \\
& \xi + b(x,y,\xi), \\
& (x-iy) \xi + {(x-iy)}^2 + r(x,y,\xi)
\bigr) ,
\end{split}
\end{equation}
where $a$, $b$ are $O(2)$ and $r$ is $O(3)$.  As the mapping is a CR mapping and a local diffeomorphism, then given any
such $a$, $b$, and $r$, the image is
necessarily Levi-flat at CR points.  Therefore the set of all these
mappings gives us all type C.1 Levi-flat submanifolds.

We precompose with an automorphism of $\R^2 \times \C$ to make $b = 0$.
We cannot similarly remove $a$ as any
automorphism must have real valued first two components (the new $x$ and the
new $y$), and hence those
components can only depend on $x$ and $y$ and not on $\xi$.  So if $a$
depends on $\xi$, we cannot remove it by precomposing.

Next we notice that
we can treat $M$ as an abstract CR manifold.  
Suppose we have two equivalent submanifolds $M_1$ and $M_2$,
with $F$ being the biholomorphic map taking $M_1$ to $M_2$.
If $M_j$ is the image of a map $\varphi_j$, then
note that $\varphi_2^{-1}$ is CR on ${(M_2)}_{CR}$.
Therefore, $G = \varphi_2^{-1} \circ F \circ \varphi_1$ is CR
on ${(F \circ \varphi_1)}^{-1}\bigl({(M_2)}_{CR}\bigr)$, which is
dense in a neighbourhood of the origin of $\R^2 \times \C$ (the
CR singularity of $M_2$ is a thin set, and we pull it back by two
real-analytic diffeomorphisms).  A real-analytic diffeomorphism that
is CR on a dense set is a CR mapping.  The same argument
works for the inverse of $G$,
and therefore we have a CR diffeomorphism of $\R^2 \times \C$.
The conclusion we make is the following proposition.

\begin{prop}
If $M_j \subset \C^3$, $j=1,2$ are given by the maps $\varphi_j$
\begin{equation}
\begin{split}
(x,y,\xi) \in \R^2 \times \C \overset{\varphi_j}{\mapsto}
\bigl(
& x+iy + a_j(x,y,\xi), \\
& \xi + b_j(x,y,\xi), \\
& (x-iy) \xi + {(x-iy)}^2 + r_j(x,y,\xi)
\bigr) ,
\end{split}
\end{equation}
and $M_1$ and $M_2$ are locally biholomorphically (resp.\ formally)
equivalent at $0$, then there exists local biholomorphisms (resp.\ formal equivalences)
$F$ and $G$ at $0$,
with $F(M_1) = M_2$, $G(\R^2 \times \C) = \R^2 \times \C$ as germs
(resp.\ formally) and
\begin{equation}
\varphi_2 = F \circ \varphi_1 \circ G .
\end{equation}
\end{prop}

In other words, the proposition states that if we find a normal form
for the mapping we find a normal form for the submanifolds.  Let us prove
that the proposition also works formally. 

\begin{proof}
We have to prove that $G$ restricted to $\R^2 \times \C$ is CR, that is, 
$\frac{\partial G}{\partial \bar{\xi}} = 0$.  Let us consider
\begin{equation}
\varphi_2 \circ G = F \circ \varphi_1 .
\end{equation}
The right hand side does not depend on $\bar{\xi}$ and thus the left hand
side does not either.  Write $G = (G^1,G^2,G^3)$.  Let us write $b = b_2$
and $r = r_2$
for simplicity.  Taking derivative of $\varphi_2 \circ G$ with respect
to $\bar{\xi}$ we get:
\begin{equation}
\begin{aligned}
& G^1_{\bar{\xi}} + 
i G^2_{\bar{\xi}} + 
a_x(G) G^1_{\bar{\xi}} +
a_y(G) G^2_{\bar{\xi}} +
a_\xi(G) G^3_{\bar{\xi}} = 0, \\
& G^3_{\bar{\xi}} + 
b_x(G) G^1_{\bar{\xi}} +
b_y(G) G^2_{\bar{\xi}} +
b_\xi(G) G^3_{\bar{\xi}} = 0, \\
&
(G^1_{\bar{\xi}} - i G^2_{\bar{\xi}}) G^3
+
(G^1 - i G^2) G^3_{\bar{\xi}} + 
2 (G^1 - i G^2)
(G^1_{\bar{\xi}} - i G^2_{\bar{\xi}})
\\
& \qquad
+
r_x(G) G^1_{\bar{\xi}} +
r_y(G) G^2_{\bar{\xi}} +
r_\xi(G) G^3_{\bar{\xi}} = 0 .
\end{aligned}
\end{equation}
Suppose that the homogeneous parts of $G^j_{\bar{\xi}}$ are zero for all
degrees up to degree $d-1$.  If we look at the degree $d$ homogeneous parts
of the first two equations above we immediately note that it must be that
$G^1_{\bar{\xi}} + i G^2_{\bar{\xi}} = 0$ and
$G^3_{\bar{\xi}} = 0$ in degree $d$.  We then look at the degree $d+1$ 
part of the third equation.  
Recall that  $[\cdot]_{d}$ is the degree $d$
part of an expression.  We get 
\begin{equation}
{[G^1_{\bar{\xi}} - i G^2_{\bar{\xi}}]}_{d}
{[G^3 + 2 G^1 - i 2 G^2]}_{1} = 0 .
\end{equation}
As $G$ is an automorphism we cannot have the linear terms be linearly
dependent and hence
$G^1_{\bar{\xi}} =  G^2_{\bar{\xi}} = 0$ in degree $d$.  We finish
by induction on $d$.
\end{proof}

Using the proposition we can restate the result of
Theorem~\ref{thm:themixedholform} using the parametrization.

\begin{cor}
A real-analytic Levi-flat type C.1 submanifold $M \subset \C^3$ is
biholomorphically equivalent to the quadric $M_{C.1}$ if and only if
the mapping giving $M$ is equivalent to a mapping of the form
\begin{equation}
(x,y,\xi) \in \R^2 \times \C \mapsto
\bigl(
x+iy,
\quad
\xi,
\quad
(x-iy) \xi + {(x-iy)}^2 + r(x,y,\xi)
\bigr) .
\end{equation}
\end{cor}

That is, $M$ is equivalent to $M_{C.1}$ if and only if we can get rid of the
$a(x,y,\xi)$ via pre and post composing with automorphisms.  The proof of
the corollary follows
as a submanifold that is realized by this map
must be of the form
$w = \bar{z}_1z_2 + \bar{z}_1^2 + \rho(z_1,\bar{z}_1,z_2)$ and
we apply Theorem~\ref{thm:themixedholform}.


  We have seen that the involution $\tau$ on $M$, in particular
when $M$ is the quadric,  is useful to compute the automorphism group and to construct Levi-flat submanifolds of type $C.1$. 
We will also need to deal with power series in $z,\bar z, \xi$. Thus we
extend $\tau$, which is originally defined on $\C^2$, as follows
\begin{equation}
\sigma(z,\bar z,\xi)=(z,-\bar z-\xi,\xi).
\end{equation}
Here $z,\bar z,\xi$ are treated as independent variables. 
Note that $z,\xi,w=\bar z\xi+\bar z^2$ are invariant by $\sigma$, while $\eta=\bar z+\frac{1}{2}\xi$ is skew invariant by $\sigma$. 
A power series in $z,\bar z,\xi$ that is invariant by $\sigma$ is precisely a power series in $z,\xi,w$.
In general, a power series $u$ in $z,\bar z,\xi$ admits a unique decomposition
\begin{equation}
u(z,\bar z,\xi)=u^+(z,\xi,w)+\eta u^-(z,\xi,w).
\end{equation}

First we introduce degree for power series $u(z,\bar z,\xi)$  and weights for
power series  $v(z,\xi,w)$.  As usual we assign degree
 $i+j+k$ to the monomial $z^i\bar z^j\xi^k$. We assign weight $i+j+2k$
to the monomial $z^i\xi^jw^k$. For simplicity, we will call them  weight in both situations. 
 Let us also denote
\begin{equation}
[u]_d(z,\bar z,\xi)=\sum_{i+j+k=d}u_{ijk}z^i\bar z^j\xi^k, \quad [v]_d(z,\xi,w)=\sum_{i+j+2k=d}v_{ijk}z^i\xi^jw^k. 
\end{equation}
Set $[u]_{i}^j=[u]_i+\cdots+[u]_j$ and $[v]_i^j=[v]_i+\cdots+[v]_j$ for $i\leq j$. 
%

\begin{thm} \label{thm:formalnormformC3}
Let $M$ be a
real-analytic Levi-flat type C.1 submanifold in $\C^{3}$. There exists a formal biholomorphic map transforming $M$ into 
 the image of
\begin{equation}
\hat\varphi(z,\bar{z},\xi)=\bigl(z+A(z,\xi, w)w\eta, \xi,w\bigr)
\end{equation}
with  $\eta=\bar z+\frac{1}{2}{\xi}$ and $w=\bar z\xi+\bar z^2$.  Suppose further that $A\not\equiv0$.
  Fix $i_*,j_*,k_*$  such that $j_*$ is the largest integer satisfying
$A_{i_*j_*k_*}\neq0$ and $i_*+j_*+2k_*=s$.  Then we can achieve 
\begin{equation}
  A_{i_*(j_*+n)k_*}=0, \quad n=1,2,\ldots. 
\end{equation}
  Furthermore, the power series $A$
 is uniquely determined up to  the transformation
\begin{equation}
 A(z,\xi,w)\to \bar c^{3}A(cz,\bar c\xi,\bar c^2w), \quad c\in\C\setminus\{0\}.
\end{equation}
In the above normal form with $A\not\equiv0$, the group of formal
biholomorphisms that preserve the normal form consists of
dilations 
\begin{equation}
(z,\xi,w)\to (\nu z,\bar\nu\xi,\bar\nu^2w)
\end{equation}
satisfying  $\bar\nu^{3}A(\nu z,\bar\nu\xi,\bar\nu^2w)=A(z,\xi,w)$.
\end{thm}

\begin{proof}
It will be convenient to write the CR diffeomorphism  $G$  of $\R^2\times \C$
as $(G_1,G_2)$ where $G_1$ is complex-valued and depends on $z,\bar z$,
while $G_2$ depends on $z,\bar z,\xi$.  Let $M$ be the image of a mapping $\varphi$ defined by
\begin{equation}
\begin{aligned}
(z,\bar{z},\xi) \overset{\varphi}{\mapsto}
\bigl(
& z + a(z,\bar z,\xi),
\\
& \xi,
\\
& \bar z \xi + {\bar z}^2 + r(z,\bar z,\xi)
\bigr)
\end{aligned}
\end{equation}
with $a=O(2), r=O(3)$.  We want to find a formal biholomorphic map $F$ of $\C^3$
and a formal CR diffeomorphism $G$ of $\R^2\times \C$ such that
\begin{equation}
F\hat\varphi G^{-1}=\varphi
\end{equation}
with $\hat{\varphi}$ in the normal form.

To simplify the computation, we will first achieve a preliminary normal form where  $r=0$ and the function $a$ is skew-invariant
by $\sigma$.
For the  preliminary normal form we will only apply $F, G$ that are tangent to the
identity. We will then use the general $F, G$ to obtain the final normal form.

Let us assume that $F, G$ are tangent to the identity.
 Let $M=F\bigl(\hat\varphi(\R^2\times\C)\bigr)$  where  $\hat \varphi$
is determined by $\hat a, \hat r$. We write
\begin{equation}
F=I+(f_1,f_2,f_3), \quad G=I+(g_1,g_2).
\end{equation}
The $\xi$ components in $\varphi G=F\hat\varphi$ give us
\begin{equation}\label{}
g_2(z,\bar z,\xi)=f_2\bigl(z+\hat a(z,\bar z,\xi), \xi, \bar z\xi+\bar
z^2+\hat r(z,\bar z,\xi)\bigr).\label{g2zz}
\end{equation}
Thus, we are allowed to define $g_2$ by the above identity for any choice of $f_2=O(2)$.
Eliminating $g_2$   in other components of $\varphi
G=F\hat \varphi$, we obtain
\begin{align}\label{f1vg}
f_1\circ\hat\varphi-g_1&=  a\circ G-\hat a,\\
f_3\circ\hat\varphi-\bar zf_2\circ\hat\varphi&=r\circ G-\hat
r+2\eta\tilde g_1+\tilde g_1f_2\circ\hat \varphi+\tilde g_1^2,
\label{f3vz}
\end{align}
where $\tilde g_1(z,\bar z)=\bar g_1(\bar z,z)$ and 
\begin{equation}
(a\circ G)(z,\bar z,\xi):=a\bigl(G_1(z,\bar z),\bar G_1(\bar z,z),
G_2(z,\bar z,\xi)\bigr).
\end{equation}

  Each power
series $r(z,\bar z,\xi)$ admits a unique decomposition
\begin{equation}
r(z,\bar z,\xi)=r^+(z,\xi,w)+\eta r^-(z,\xi,w),
\end{equation}
where both $r^\pm$ are invariant by $\sigma$. Note that  $r(z,\bar z,\xi)$ is a power series in $z,\xi$ and $w$, 
if and only if it is invariant by $\sigma$, i.e.\ if $r^-=0$.
We write
\begin{equation}
r^+={wt}\, (k), \quad\text{or}\quad  {wt} \, (r^+)\geq k,
\end{equation}
if $r^+_{abc}=0$ for $a+b+2c<k$. Define $r^-=wt  (k)$ analogously
and write $\eta r^-={wt}\,(k)$ if $r^-={wt}\,(k-1)$. We write
$r={wt}\,(k)$ if $(r^+,\eta r^-)={wt}\,(k)$. Note that
\begin{equation}
r=O(k)\Rightarrow r={wt}\,(k); \quad wt \, (rs)\geq wt\, (r)+wt\,
(s).
\end{equation}

The power series in $z,\bar z$ play a special role in describing normal forms. Let us define $T^\pm$ via
\begin{equation}
u(z,\bar z)=(T^+u)(z,\xi,w)+(T^-u)(z,\xi,w)\eta.
\end{equation}
Let $S^+_k$ (resp.\ $S^-_k$) be spanned by monomials  in $z,\bar z,\xi$
 which have weight $k$ and are invariant (resp.\ skew-invariant)
 by $\sigma$. Then the range of $\eta T^-$ in $S_k^-$ is a linear subspace $R_k$.  We decompose
\begin{equation}
 S_k^-=R_k\oplus(S_k^-\ominus R_k).
\end{equation}
The decomposition is of course  not unique. We will take
\begin{equation}
 S_k^-\ominus R_k=\bigoplus_{a+b+2c=k-1, c>0} \C z^a\xi^bw^c\eta.
\end{equation}
Here we have used $\eta=\bar{z}+\frac{1}{2}\xi$, $\eta^2=w+\frac{1}{4}\xi^2$,  and 
\begin{gather}
T^+u(z,\xi, w)=\sum_{i,j\geq0}\sum_{0\leq \alpha\leq j/2} u_{ij} \binom{j}{2\alpha}z^i(w+\frac{1}{4}\xi^2)^\alpha(-\frac{1}{2}\xi)^{j-2\alpha},\\
T^-u(z,\xi,w)=\sum_{i\geq0,j>0}\sum_{0\leq \alpha<j/2} u_{ij}\binom{j}{2\alpha+1}z^i(w+\frac{1}{4}\xi^2)^\alpha(-\frac{1}{2}\xi)^{j-2\alpha-1}.
\end{gather}
In particular, we have
\begin{equation}\label{uzx0}
T^-u(z,\xi,0)=\sum_{i\geq0,j>0}(-1)^{j-1}u_{ij} z^i \xi^{j-1}.
\end{equation}
This shows that
\begin{gather}
\label{uzx0+}
T^-u(z,\xi,0)=\frac{1}{-\xi}\bigl(u(z,-\xi)-u(z,0)\bigr).
\end{gather}
We are ready to show that under the condition that $g_1(z,\bar z)$ has no pure holomorphic terms, there exists
a unique $(F,G)$ which is tangent to the identity such that $\hat r=0$ and
\begin{equation}
\hat a\in\mathcal N:=\bigoplus\mathcal N_k, \quad \mathcal N_k:=
S_k^-\ominus R_k.
\end{equation}

 We start with terms of weight $2$ in \eqref{f1vg}-\eqref{f3vz} to get
 \begin{gather}\label{f1vg2}
[f_1]_2-[g_1]_2=[a]_2-\eta[\hat a^-]_1,\\
[f_3]_2=0.
\label{f3vz2}
\end{gather}
Note that $f_j^-=0$.     The first identity implies that
 \begin{equation}
 [f_1]_2-[T^+g_1]_2=[a^+]_2, \quad [T^-g_1]_1=[\hat a^-]_1-[a^-]_1.
 \end{equation}
 The first equation is solvable  with kernel defined by
 \begin{equation}\label{f11q}
 [f_1]_k-[T^+g_1]_k=0 
  \end{equation}
for $k=2$.  This shows that $[g_1]_2$ is still arbitrary and
 we use it to achieve
\begin{equation}
\eta [\hat a^-]_1\in S_2^-\ominus R_2=\{0\}.
\end{equation}
Then the kernel space is defined
 by   \eqref{f11q} and \begin{equation}
\label{Tmg1}
[g_1(z,\bar{z})-g_1(z,0)]_k=0
\end{equation}
with $k=2$.  In particular, under the restriction
 \begin{equation}\label{g1z0}
 [g_1(z,0)]_k=0,
 \end{equation}
 for $k=2$, we have achieved $\hat a^-\in \mathcal N_2$
  by unique $[f_1]_2, [g_1]_2, [f_2]_1, [f_3]_2$.
 By induction, we verify that if \eqref{g1z0} holds for all $k$, we determine
 uniquely  $[f_1]_k, [g_1]_k$ by normalizing $[\hat a]_k\in\mathcal N_k$. We then determine $[f_2]_k, [f_3]_{k+1}$ uniquely to normalize
 $[\hat r]_{k+1}=0$.  For the details, let us find formula for the solutions. 
 We rewrite \eqref{f1vg} as
 \begin{gather}
 T^-g_1=-(a\circ G-\hat a-f_1\circ\hat\varphi)^-,\label{t-g1}\\
( f_1\circ\hat\varphi)^+=(a\circ G-\hat a)^++T^+g_1.\label{f1hv}
 \end{gather}
Using \eqref{uzx0},  we can solve
 \begin{equation}
 (-1)^{j-1}g_{1,ij}=
 -({(a\circ G)}^-)_{i(j-1)0}, \quad j\geq1, \quad i+j=k. 
\label{eq15}
\end{equation}
Then we have
\begin{gather}
(\hat a^-)_{ij0}=0, \quad i+j=k-1; \\
(\hat a^-)_{ij m}=((a\circ G-f_1\circ\hat\varphi+g_1)^-)_{ijm}, \quad m\geq1,  i+j+m=k-1. 
\end{gather}
Note that $
- [g_1]_k(z,-\bar z)=\bar z[(a\circ G-\hat a)^-]_{k-1}(z,\bar z,0)$. We obtain
\begin{gather}
[g_1]_k(z,\bar z)=\bar z[(a\circ G-\hat a)^-]_{k-1}(z,-\bar z,0).\label{eq16}
\end{gather}
Having determined $[g_1]_k$, we take 
\begin{gather}
[ f_1]_k=[(a\circ G-\hat a+g_1)^+]_k.
  \end{gather}
 We then solve \eqref{f3vz} by taking
 \begin{gather}\label{f2k=}
[ f_2]_k=[E^-]_k, \quad [f_3]_{k+1}=[(E-\frac{1}{2}\xi f_2)^+]_{k+1},\\
 E:=r\circ G-\hat
r+2\eta\tilde g_1+\tilde g_1f_2\circ\hat \varphi+\tilde g_1^2.\label{e=rg}
 \end{gather}
We have achieved the preliminary normalization. 
 
Assume now that 
\begin{gather} 
\varphi(z,\bar z,\xi)=(z+a^-(z,\xi,w)\eta, \xi,w), \quad \hat\varphi (z,\bar z,\xi)=(z+\hat a^-(z,\xi,w)\eta, \xi,w)
\end{gather}
are in the preliminary normal form, i.e.
\begin{equation}
w|a^-(z,\xi,w), \quad w|\hat a^-(z,\xi,w). 
\end{equation}
Let us assume that
\begin{equation}
a^-(z,\xi,w)=wt (s), \quad [a^-]_s\not\equiv0; \quad \hat a^-(z,\xi,w)=wt(s). 
\end{equation}
We assume that $\varphi G=F\hat\varphi$ with
\begin{gather}
F(z,\xi,w)=I+(f_1,f_2,f_3),\\
G(z,\bar{z},\xi)=(z+g_1(z,\bar{z}), \xi+g_2(z,\bar{z},\xi)).
\end{gather}
Here $f_i,g_j$ start with terms of weight and order at least $2$. 
In particular, we have
\begin{equation}
f_i=wt(N), \quad g_i=wt (N),\quad i=1,2;   \quad f_3=wt(N'); \quad N'\geq N\geq2. 
\end{equation}
Set $(P,Q,R):=\varphi G$. Using $N\geq2$, $s\geq2$, and the Taylor theorem, we obtain
\begin{align}
P&=z+g_1(z,\bar{z})+a^-(z,\xi,w)\eta+a^-(z,\xi,w)(\bar g_1(\bar z,z)+\frac{1}{2}g_2(z,\bar z,\xi))\\
&\quad +\eta\nabla a^-(z,\xi,w)\cdot \Bigl(g_1(z,\bar z), g_2(z,\bar z,\xi), 
(
\xi+2\bar z)\bar g_1(\bar z,z)+\bar zg_2(z,\bar z,\xi)\Bigr)\\
&\quad+wt(s+N+1),\\
Q&=\xi+g_2(z,\bar z,\xi),\\
R&=w+(2\bar z+\xi)\bar g_1(\bar z,z)+\bar zg_2(z,\bar z,\xi)+wt(2N).
\end{align}
We also have $(P,Q,R)=F\hat\varphi$. Thus
\begin{align}
P&=z+\hat a^-(z,\xi,w)\eta+f_1(z,\xi,w)+\partial_zf_1(z,\xi,w)\hat a^-(z,\xi,w)\eta+wt(N+s+1),\\
Q&=\xi+f_2(z,\xi,w)+\partial_zf_2(z,\xi,w)\hat a^-(z,\xi,w)\eta+wt(N+s+1),\\
R&=w+f_3(z,\xi,w)+\partial_zf_3(z,\xi,w)\hat a^-(z,\xi,w)\eta+wt(N'+s+1).\label{f3g1}
\end{align}
We will use the above 6 identities for $P,Q,R$ in two ways. 
First we use their lower order terms to  get
\begin{gather}\label{f1g1}
f_1(z,\xi,w)=g_1(z,\bar z)+( a^-(z,\xi,w)-\hat a^-(z,\xi,w))\eta+wt(N+s),\\
\quad f_2(z,\xi,w)=g_2(z,\bar z,\xi)+wt(N+s), \\  f_3(z,\xi,w)=(2\bar z+\xi)\bar g_1(\bar z,z)+\bar zg_2(z,\bar z,\xi)
+wt(2N)+wt(N'+s).\label{f3zx}
\end{gather}
Hence, we can take $N'=N+1$.  By \eqref{f1g1} and the preliminary normalization, we first know that
\begin{gather}
\hat a=a+wt(N+s-1), \\
 f_1(z,\xi,w)=b(z)+wt (N+s), \quad g_1(z,\bar z)=b(z)+wt (N+s). \label{f1g1n}
\end{gather}
We compose \eqref{f3zx} by $\sigma$ and then take the difference of the two equations to get
\begin{gather}
f_2(z,\xi,w)=-\bar b(\bar z)-\bar b( -\bar z-\xi)+wt(2N-1)+wt(N+s), \\
 f_3(z,\xi,w)=
-\bar z\bar b(-\bar z-\xi)
+(\bar z+\xi)\bar b(\bar z)+wt(2N)+wt(N+s+1).
\end{gather}
Here we have used $N'=N+1$. 
Let $b(z)=b_Nz^N+wt(N+1)$. 
Therefore, we have
\begin{gather}
g_2(z,\bar z,\xi)=-\bar b_N(\bar z^N+(-\bar z-\xi)^N)+wt(N+1),\\
\bar g_1(\bar z,z)+\frac{1}{2}g_2(z,\bar z,\xi)=
\eta\bar b_N\sum\bar z^i(-\bar z-\xi)^{N-1-i}+wt(N+1),\\
(2\bar z+\xi)\bar g_1(\bar z,z)+\bar zg_2(z,\bar z,\xi)=
\bar b_N(\bar z^{N-1}+(-\bar z-\xi)^{N-1})w+wt(N+2).
\end{gather}

Next, we use the two formulae for $P$  and \eqref{f1g1n} to get the identity in higher weight:
\begin{align}\label{hata-}
\hat a^-&=a^-+g_1^-+Lb_N+wt(N+s), \quad f_1-g_1^+=wt (N+s+1). 
\end{align}
Here we have used $f_1^-=0$ and
\begin{align}
Lb_N(z,\xi,w)&:=-Nb_Nz^{N-1} [a^-]_s(z,\xi,w)-[a^-]_s(z,\xi,w)\bar b_N\sum_i\bar z^i(-\bar z-\xi)^{N-1-i}\nonumber\\
&\quad +\nabla [a^-]_s\cdot \Bigl(b_Nz^N, -\bar b_N(\bar z^N+(-\bar z-\xi)^N),
\bar b_Nw(\bar z^{N-1}+(-\bar z-\xi)^{N-1})\Bigr).
\end{align}
Recall that $w|a^-$ and $w|\hat a^-$. We also have that
$
w|Lb_N(z,\xi,w)
$
and $Lb_N$ is homogenous 
in weighted variables and of weight $N+s-1$. 
This shows that $[g_1^-(z,\xi,0)]_{N+s-1}=0$. By \eqref{uzx0}, we get
\begin{equation}
[g_1(z,\bar z)]_{N+s}=[g_1(z,0)]_{N+s}, \quad [\hat a^-]_{s+N-1}=[ a^-]_{s+N-1}+Lb_N.
\end{equation}

Let us make some  observations. First, $Lb_N$ depends only on $b_N$ and it does not depend on coefficients of $b(z)$ of degree larger than $N$. 
We observe that the first identity says that all coefficients of $[g_1]_{N+s}$ must be zero, except that the coefficient
$g_{1,(N+s)0}$ is arbitrary. On the other hand $Lb_N$, which has weight $N+s-1$, 
depends only on $g_{1,N0}$, while $N+s-1>N$. 
Let us assume for the moment that we have $Lb_N\neq0$ for all $b_N\neq0$. We will then choose a suitable complement subspace ${\mathcal N}^*_{N+s-1}$
in the space of weighted homogenous polynomials in $z,\xi,w$ of weight $N+s-1$
for $Lb_N$.  Then $\hat a^-\in w\sum_{N>1}{\mathcal N}^*_{N+s-1}$ will be the required  normal form. 
The normal form will be obtained by the following procedures: Assume that $\varphi$ is not formally equivalent to the quadratic mapping in the preliminary normalization. 
We first achieve the preliminary normal form by a mapping $F^0=I+(f_1^0,f_2^0,f_3^0)$ and $G^0=I+(g_1^0,g_2^0)$ which are tangent to the identity. We can make
$F^0,G^0$ to be unique by requiring $f^1_1(z,0)=0$.  Then $a$ is normalized such that $\hat a=\hat a^-\eta$ with $[\hat a^-]_s$ being non-zero
 homogenous part of the lowest weight. We may assume that $[a]_{s+1}=[\hat a]_{s+1}$. 
Inductively, we choose $f^1_{1,N00}$  ($N=2, 3, \ldots$) to achieve $[\hat a^-]_{N+s-1}\in w{\mathcal N}^*_{N+s-1}$.
In this step for a given $N$, we determine mappings $F^1=I+(f_1^1,f_2^1,f_3^1)$ and $G^1=I+(g_1^1,g_2^1)$ by requiring that $f_1^1(z,\xi,w)$ contains only
one term $\xi^N$, while $f_1^1,f_2^1,g_1^1,g_2^1$ have weight at most $N$ and $f_3^1$ has weight at most $N+1$. 
In the process, we also show that $[f_1^1(z,\xi,w)]_2^{N+s}$ depends only on $z$, if we do not want to impose the restriction on $f_1^1$.
Moreover, the coefficient of $\xi^{N+s-1}$ of $f_1^1$   can still be arbitrarily chosen without changing the normalization
achieved for $[\hat a^-]_{N+s-1}$ via $[f_1^1]_{N}$.  However, by achieving  $[\hat a^-]_{N+s-1}\in w{\mathcal N}^*_{N+s-1}$ via $F^1,G^1$, we may destroy the
preliminary normalization achieved via $F_0,G_0$. We will then restore the preliminary normalization via $F^2=I+(f_1^2,f_2^2,f_3^2), G^2=I+(g_1^2,g_2^2)$ satisfying
$g^2_1(z,0)=0$. This amounts to determining $g_1^2=g_1$ and $f_1^2=f_1$
via  \eqref{t-g1} and \eqref{f1hv} for which the terms of weight at most $N+s$ have been
determined by \eqref{hata-}, and then $f_2^2=f_2,f_3^2=f_3,g_2^2=g_2$ are determined by \eqref{f2k=}-\eqref{e=rg} and \eqref{g2zz}, respectively. 
 This allows us to repeat the procedure to achieve the normalization in any higher weight.
We will then remove the restriction that the normalizing mappings must be tangent
to the identity. This will alter  the normal form only by suitable linear dilations.

Suppose that $b_N\neq0$. Let us verify that 
\begin{equation}\label{lbn0}
Lb_N\neq0.
\end{equation}
We will also identify one of non-zero coefficients to describe the normalizing condition on $\hat a$. 
We write the two invariant polynomials
\begin{gather}
\label{zbn+}
\bar z^N+(-\bar z-\xi)^N
=\lambda_N\xi^N+\sum_{j<N}p_{ijk}z^i\xi^jw^k,
\\
\label{zbn++}
\sum_i\bar z^i(-\bar z-\xi)^{N-1-i}=\lambda_{N-1}'\xi^{N-1}+\sum_{j<N-1}q_{ijk}z^i\xi^jw^k.
\end{gather}
If we plug in $w=\bar{z}^2+\bar{z}\xi$ we obtain a polynomial identity in
the variables $z,\bar{z},\xi$.
\begin{gather}
\bar z^N+(-\bar z-\xi)^N
=\lambda_N\xi^N+\sum_{j<N}p_{ijk}z^i\xi^j{(\bar{z}^2+\bar{z}\xi)}^k,
\\
\sum_i\bar z^i(-\bar
z-\xi)^{N-1-i}=\lambda_{N-1}'\xi^{N-1}+\sum_{j<N-1}q_{ijk}z^i\xi^j{(\bar{z}^2+\bar{z}\xi)}^k.
\end{gather}
If we set $\bar{z} = z = 0$, we obtain that
\begin{equation}
\lambda_N = \lambda'_N = {(-1)}^N .
\end{equation}

 Recall that $j_*$ is the largest integer such that $(a^-)_{i_*j_*k_*}\neq0$ and $i_*+j_*+2k_*=s$.  
Since $w|[a^-]_s$,  then $k_*>0$.  We obtain
\begin{equation}
(Lb_N)_{i_*(j_*+N-1)k_*}=(a^-)_{i_*j_*k_*}\bar b_N(-\lambda_{N-1}'-j_*\lambda_{N-1}+k_*\lambda_N)\neq0.
\end{equation}
 Therefore, we can achieve
\begin{equation}
(\hat a^-)_{i_*(j_*+n)k_*}=0, \quad n=1,2,\ldots. 
\end{equation}
This determines uniquely all $b_2, b_3, \ldots.$ 
 
We now remove the restriction that $F$ and $G$ are tangent to the identity. Suppose that both $\varphi$ and $\hat\varphi$ are in the normal form.
Suppose that $F\varphi=\hat\varphi G$.  Then looking at the quadratic terms, we know that the linear parts $F,G$ 
must be dilations. In fact, the linear part of $F$ must be the linear automorphism of the quadric.  Thus the linear parts of $F$ and $G$ have
 the forms
\begin{equation}
G'\colon(z,\xi)=(\nu z,\bar \nu\xi), \quad F'(z,\xi,w)=(\nu z,\bar \nu\xi,\bar \nu^2w). 
\end{equation}
Then $(F')^{-1}\hat\varphi G'$ is still in the normal form. Since $(F')^{-1}F$ is holomorphic and $(G')^{-1}G$ is CR, by the uniqueness
of the normalization, we know that $F'=F$ and $G'=G$. 
Therefore,  $F$ and $G$ change the normal form $a^-$
as follows
\begin{equation}
a^-(z,\xi,w)= \bar \nu \hat a^-(\nu z,\bar \nu\xi,\bar \nu^2w), \quad \nu\in\C\setminus\{0\}. 
\end{equation}
When $[\hat a^-]_s=[a^-]_s\neq0$, we see that $|\nu|=1$. Therefore, the formal automorphism group is discrete or one-dimensional. 
\end{proof}
In~\cite{Coffman:crosscap}, Coffman used an analogous method of even/odd function decomposition to obtain a quadratic normal form for non Levi-flat real analytic
$m$-submanifolds in $\C^n$ with an CR singularity satisfying certain non-degeneracy conditions,
provided $\frac{3(n+1)}{2} \leq m<n$.
He was able to achieve the convergent normalization by a rapid iteration method. 
Using the above decomposition of invariant and skew-invariant functions of the involution $\sigma$, one might achieve a convergent
solution for approximate equations when $M$ is formaly equivalent to the quadric. 
However, when the iteration is employed,  each new
CR mapping $\hat\varphi$  might only be defined on a domain that is
proportional to that  of the previous $\varphi$ in a \emph{constant} factor. 
 This is significantly  different from the situations of Moser~\cite{Moser85} and Coffman~\cite{Coffman:crosscap},
\cite{Coffman:unfolding}, where rapid iteration methods are applicable. 
 Therefore, even if $M$ is formally equivalent to the quadric, we do not know if they are holomorphically
equivalent.


\section{Instability of Bishop-like submanifolds}
\label{sec:bishopexamples}

Let us now discuss stability of Levi-flat submanifolds under small
perturbations that keep the submanifolds Levi-flat, in particular
we discuss which quadratic invariants are stable when moving from point to
point on the submanifold.  The only
stable submanifolds are A.$n$ and C.1.  The Bishop-like submanifolds (or even
just the Bishop invariant) are not stable under perturbation, which we
show by constructing examples.

\begin{prop} \label{prop:instability}
Suppose that $M \subset \C^{n+1}$, $n \geq 2$, is a connected
real-analytic
real codimension 2
submanifold that has a non-degenerate CR singular at the origin.  $M$ can be written
in coordinates $(z,w) \in \C^{n} \times \C$ as
\begin{equation}
w = A(z,\bar{z}) + B(\bar{z},\bar{z}) + O(3),
\end{equation}
for quadratic $A$ and $B$.  
In a neighborhood of the origin all complex tangents
of $M$ are non-degenerate,  while ranks of $A,B$ are upper semicontinuous.
Suppose that $M$ is Levi-flat (that is $M_{CR}$
is Levi-flat).
The CR singular set of $M$ that is not of type B.$\frac{1}{2}$ at the origin is a
real analytic subset of $M$ of codimension at least $2$, while  the CR
singular set of $M$ that is of type  B.$\frac{1}{2}$ the origin has codimension at
least $1$.  A.$n$ has an isolated CR singular point at the origin and so
does C.1 in $\C^3$.
Let $S_0 \subset M$ be the set of CR singular points.
There is a neighborhood $U$ of the origin such that for $S=S_0\cap U$
we have the following.
\begin{enumerate}[(i)]
\item \label{thmitem:Akall}
If $M$ is of type A.$k$ for $k \geq 2$ at the origin, then it is   of type A.$j$  at each point
of $S$ for some
$j \geq k$.
\item \label{thmitem:C1dense}
If $M$ is of type C.1 at the origin, then it is of type C.1 on   $S$.
If $M$ is of type C.0 at the origin, then it is of type C.0 or C.1 on   $S$.
\item \label{thmitem:examples}
There exists an $M$ that is of type B.$\gamma$ at one point and of
C.1 at CR singular points arbitrarily near.  Similarly there exists an $M$
of type A.1 at $p \in M$ that is either of type C.1, or B.$\gamma$, at
points arbitrarily near $p$.  There also
exists an $M$ of type B.$\gamma$ at every point but where $\gamma$
varies from point to point.
\end{enumerate}
\end{prop}

\begin{proof}
First we show that 
the rank of $A$ and the rank of $B$ are lower semicontinuous on $S_0$,
without imposing Levi-flatness condition.
Similarly the real dimension of the range of $A(z,\bar{z})$ is 
lower
semicontinuous on $S_0$.
Write $M$ as
\begin{equation}
w = \rho(z,\bar{z}) ,
\end{equation}
where $\rho$ vanishes to second order at 0.  If we move to a different
point of $S_0$ via an affine map $(z,w) \mapsto (Z+z_0,W+w_0)$.  Then we have
\begin{equation}
W+w_0 = \rho(Z+z_0,\bar{Z}+\bar{z}_0) .
\end{equation}
We compute the Taylor coefficients
\begin{multline}
W =
\frac{\partial \rho}{\partial z} (z_0,\bar{z}_0) \cdot Z + 
\frac{\partial \rho}{\partial \bar{z}} (z_0,\bar{z}_0) \cdot \bar{Z} + \\
+
Z^*
\left[
\frac{\partial^2 \rho}{\partial z \partial \bar{z}} (z_0,\bar{z}_0)
\right] Z +
\frac{1}{2}
Z^t
\left[
\frac{\partial^2 \rho}{\partial z \partial z} (z_0,\bar{z}_0)
\right] Z +
\frac{1}{2}
Z^*
\left[
\frac{\partial^2 \rho}{\partial \bar{z} \partial \bar{z}} (z_0,\bar{z}_0)
\right] \bar{Z} +
O(3) .
\end{multline}
The holomorphic terms can be absorbed into $W$.  If 
$\frac{\partial \rho}{\partial \bar{z}} (z_0,\bar{z}_0) \cdot \bar{Z}$
is nonzero, then
this complex defining function has a linear term in $W$ and linear term in
$\bar{Z}$ and the submanifold is CR at this point.
Therefore
the set of complex tangents of $M$ is defined by
\begin{equation}
\frac{\partial \rho}{\partial \bar{z}} =0
\end{equation}
and each complex tangent point is non-degenerate.
At a complex tangent point at the origin,
$A$ is given by
$\left[ \frac{\partial^2 \rho}{\partial z \partial \bar{z}} (z_0,\bar{z}_0) \right]$
and 
$B$ is given by $\frac{1}{2} \left[
\frac{\partial^2 \rho}{\partial \bar{z} \partial \bar{z}} (z_0,\bar{z}_0) \right]$.
In particular these matrices change continuously as we move along $S$.
We first conclude that all CR singular points of $M$ in a neighborhood of the origin
are non-degenerate.
Further holomorphic transformations act on $A$ and $B$ using
Proposition~\ref{prop:astensors}.
Therefore the ranks of $A$ and $B$
as well as
the real dimension of the
range of $A(z,\bar{z})$ 
 are lower semicontinuous on $S_0$ as claimed.  Furthermore as $M$ is
real-analytic, the points where the rank drops lie on a real-analytic
subvariety of $S_0$, or in other words a thin set.  Let $U$ be a small
enough neighbourhood of the origin so that $S = S_0 \cap U$ is connected.

Imposing the condition that $M$ is Levi-flat, we apply
Theorem~\ref{thm:quadratic}.
By a simple computation,
unless $M$ is of type B.$\frac{1}{2}$,
the set of complex tangents of $M$ has codimension at least $2$; and A.$n$
has isolated CR singular point and so does C.1 in $\C^3$. 
The item \eqref{thmitem:Akall} follows as A.$k$ are the only types
where the rank of $B$ is greater than 1, and the theorem says $M$ must be
one of these types.
For \eqref{thmitem:C1dense} note that since $A$ is of rank 1 when $M$
as C.$x$ at a point, $M$ cannot be of type A.$k$ nearby.  If $M$ is of type
C.1 at a point then the range of $A$ must be of real dimension 2 in a
neighbourhood, and hence on this neighbourhood $M$ cannot be of type
B.$\gamma$.

The examples proving item \eqref{thmitem:examples} are
given below.
\end{proof}


\begin{example}
Define $M$ via
\begin{equation}
w = \abs{z_1}^2 + \gamma \bar{z}_1^2 + \bar{z}_1z_2z_3 .
\end{equation}
It is Levi-flat by Proposition~\ref{prop:imagelf}.
At the origin $M$ is a type B.$\gamma$, but at a point
where $z_1 = z_2 = 0$ but $z_3 \not= 0$, the submanifold is
CR singular and it is of type C.1.
\end{example}

\begin{example}
Similarly if we define $M$ via
\begin{equation}
w = \bar{z}_1^2 + \bar{z}_1z_2 z_3 ,
\end{equation}
we obtain a CR singular Levi-flat $M$ that is A.1 at the origin, but C.1 at nearby CR singular
points.
\end{example}

\begin{example}
If we define $M$ via
\begin{equation}
w = \gamma \bar{z}_1^2 +  \abs{z_1}^2 z_2 ,
\end{equation}
then $M$ is a CR singular Levi-flat type A.1 submanifold
at the origin, but type B.$\gamma$
at points where $z_1 = 0$ but $z_2 \not= 0$.
\end{example}

\begin{example} \label{example:interpolation}
The Bishop invariant can vary from point to point.
Define $M$ via
\begin{equation}
w = \abs{z_1}^2 + \bar{z}_1^2 \bigl(\gamma_1 (1-z_2) + \gamma_2 z_2 \bigr) ,
\end{equation}
where $\gamma_1 , \gamma_2 \geq 0$.
It is not hard to see that $M$ is Levi-flat.  Again it is an image
of $\C^2 \times \R^2$ in a similar way as above.

At the origin, the submanifold is Bishop-like with Bishop invariant $\gamma_1$.
When $z_1=0$ and $z_2 = 1$, the Bishop invariant is $\gamma_2$.  In fact
when $z_1=0$, the Bishop invariant at that point
is
\begin{equation}
\abs{\gamma_1 (1-z_2) + \gamma_2z_2} .
\end{equation}

Proposition~\ref{prop:imagelf} says that this submanifold possesses a
real-analytic foliation extending the Levi-foliation through the singular
points.
Proposition~\ref{prop:flatfolmflds} says that if a foliation on
$M$ extends to a (nonsingular) holomorphic foliation, then the submanifold would
be a simple product of a Bishop submanifold and $\C$.  Therefore,
if $\gamma_1 \not= \gamma_2$ then the Levi-foliation on $M$
cannot extend to a holomorphic foliation of a neighbourhood of $M$.
\end{example}


\def\MR#1{\relax\ifhmode\unskip\spacefactor3000 \space\fi%
  \href{http://www.ams.org/mathscinet-getitem?mr=#1}{MR#1}}

\end{document}